\documentclass[aop,MSNbibl,citesort,dvips]{arximspdf}
\usepackage{graphicx}

\doi{10.1214/10-AOP578}
\volume{39}
\issue{3}
\pubyear{2011}
\firstpage{1061}
\lastpage{1096}

\makeatletter

\newcommand{\ba}{\mathbf{A}}
\newcommand{\bb}{\mathbf{B}}
\newcommand{\bc}{\mathbf{C}}
\newcommand{\be}{\mathbf{E}}
\newcommand{\bx}{\mathbf{X}}
\newcommand{\bh}{\mathbf{H}}
\newcommand{\bk}{\mathbf{k}}
\newcommand{\bmm}{\mathbf{m}}
\newcommand{\bn}{\mathbf{n}}
\newcommand{\der}{\delta}
\newcommand{\hbb}{\hat{\mathbf{B}}}
\newcommand{\ott}{[0,T]}
\newcommand{\tpi}{\tilde\pi}
\newcommand{\zc}{\stackrel{\mathcal{Z}\mathcal{C}_2}{=}}

\newcommand{\cac}{{\mathcal C}}
\newcommand{\cn}{{\mathcal N}}
\newcommand{\cs}{{\mathcal S}}
\newcommand{\cZ}{{\mathcal Z}}

\newcommand{\al}{\alpha}
\newcommand{\ga}{\gamma}
\newcommand{\ep}{\varepsilon}
\newcommand{\ka}{\kappa}
\newcommand{\si}{\sigma}
\newcommand{\vp}{\varphi}

\newcommand{\R}{{\mathbb R}}
\newcommand{\T}{{\mathbb T}}

\newtheorem{theorem}{Theorem}[section]
\newtheorem{lemma}[theorem]{Lemma}
\newtheorem{proposition}[theorem]{Proposition}

\newproclaim{remark}[theorem]{Remark}

\makeatother

\begin{document}
\begin{frontmatter}

\title{A construction of the rough path above fractional Brownian
motion using Volterra's representation}
\runtitle{Rough path above fBm}

\begin{aug}
\author[A]{\fnms{David} \snm{Nualart}\thanksref{t1}\ead
[label=e1]{nualart@math.ku.edu}} and
\author[B]{\fnms{Samy} \snm{Tindel}\corref{}\thanksref{t2}\ead
[label=e2]{tindel@iecn.u-nancy.fr}}
\runauthor{D. Nualart and S. Tindel}
\affiliation{University of Kansas and Institut {\'E}lie Cartan Nancy}
\address[A]{Department of Mathematics\\
University of Kansas\\
405 Snow Hall\\
Lawrence, Kansas\\
USA\\
\printead{e1}} 
\address[B]{Institut {\'E}lie Cartan Nancy\\
B.P. 239, 54506 Vandoeuvre-l{\`e}s-Nancy Cedex\\
France\\
\printead{e2}}
\end{aug}

\thankstext{t1}{Supported in part by the NSF Grant NSF0063812.}

\thankstext{t2}{Supported in part by the ANR Grant ECRU.}

\received{\smonth{9} \syear{2009}}
\revised{\smonth{6} \syear{2010}}

%
\begin{abstract}
This note is devoted to construct a rough path above a multidimensional
fractional Brownian motion $B$ with any Hurst parameter $H\in(0,1)$, by
means of its representation as a Volterra Gaussian process. This
approach yields some algebraic and computational simplifications with
respect to [\textit{Stochastic Process. Appl.} \textbf{120} (2010)
1444--1472],
where the construction of a rough path over
$B$ was first introduced.
\end{abstract}

%
\begin{keyword}[class=AMS]
\kwd{60H05}
\kwd{60H07}
\kwd{60G15}.
\end{keyword}
\begin{keyword}
\kwd{Rough paths theory}
\kwd{fractional Brownian motion}
\kwd{multiple stochastic integrals}.
\end{keyword}

\end{frontmatter}

\section{Introduction}\label{intro}
Rough paths analysis is a theory introduced by Terry Lyons in the
pioneering paper
\cite{Ly1} which aims to solve differential equations driven by
functions with finite $p$-variation with $p>1$, or by H\"older
continuous functions of order $\gamma\in(0,1)$.
One possible shortcut to the rough path theory is the following summary
(see \cite{FV-bk,Gu,LQ-bk} for a complete construction). Given a $\ga
$-H\"{o}lder $d$-dimensional process $X=({X}(1),\ldots,{X}(d))$ defined
on an arbitrary interval $\ott$, assume that one can define some
iterated integrals of the form
%
%
\begin{equation}\label{eq:def-Xn-intro}
\bx^{\bn}_{st}(i_1,\ldots,i_n)=
\int_{s\le u_1<\cdots<u_n\le t}
dX_{u_1}(i_1) \,dX_{u_2}(i_2) \cdots dX_{u_n}(i_n),
\end{equation}
for $0\le s<t\le T$, $n\le\lfloor1/\ga\rfloor$ and $i_1,\ldots
,i_n\in
\{1,\ldots,d\}$. As long as $X$ is a
nonsmooth function, the integral above cannot be defined rigorously in
the Riemann sense (and not even in the Young sense if $\ga\le1/2$).
However, it is reasonable to assume that some elements $\bx^{\bn}$ can
be constructed, sharing the following three properties with usual
iterated integrals (here and in the sequel, we denote by
$\cs_{k,T}
=\{ (u_1,\ldots, u_k)\dvtx0\le u_1 < \cdots< u_k \le T\}$ the $k$th order
simplex on $[0,T]$):
\begin{enumerate}[(1)]
\item[(1)] \textit{Regularity}:
each component of $\mathbf{X}^{\mathbf{n}}$ is $n\ga$-H\"{o}lder continuous
[in the sense of the H\"older norm introduced in (\ref
{eq:def-norm-C2})] for all $n\le\lfloor1/\ga\rfloor$ and $\mathbf
{X}^{\mathbf{1}}_{st}=X_t-X_s$.
\item[(2)]
\textit{Multiplicativity}:
letting $(\der\mathbf{X}^{\mathbf{n}})_{sut}:=\mathbf{X}^{\mathbf
{n}}_{st}-\mathbf{X}^{\mathbf{n}}_{su}-\mathbf{X}^{\mathbf{n}}_{ut}$
for $(s,u,t)\in\cs_{3,T}$, one requires
%
%
\begin{equation}\label{eq:multiplicativity}
(\der\mathbf{X}^{\mathbf{n}})_{sut}(i_1,\ldots,i_n)=\sum_{n_1=1}^{n-1}
\mathbf{X}_{su}^{\mathbf{n_1}}(i_1,\ldots,i_{n_1}) \mathbf
{X}_{ut}^{\mathbf{n-n_1}}(i_{n_1+1},\ldots,i_n).
\end{equation}
\item[(3)] \textit{Geometricity}:
for any $n,m$ such that $n+m\le\lfloor1/\ga\rfloor$and $(s,t)\in
\cs
_{2,T}$, we have
%
%
\begin{equation}\label{eq:geom-rough-path}
\bx^{\bn}_{st}(i_1,\ldots,i_n) \bx^{\bmm}_{st}(j_1,\ldots,j_m)
=\sum_{\bar k\in\mathrm{Sh}(\bar\imath,\bar\jmath)}
\bx^{\bn+\bmm}_{st}(k_1,\ldots,k_{n+m})
,
\end{equation}
where, for two tuples $\bar\imath,\bar\jmath$, $\Sigma_{(\bar
\imath,\bar
\jmath)}$ stands for the set of permutations of the indices contained
in $(\bar\imath,\bar\jmath)$, and $\mathrm{Sh}(\bar\imath,\bar
\jmath)$ is
a subset of $\Sigma_{(\bar\imath,\bar\jmath)}$ defined by
\[
\mathrm{Sh}(\bar\imath,\bar\jmath)=
\bigl\{\si\in\Sigma_{(\bar\imath,\bar\jmath)};
\si\mbox{ does not change the orderings of } \bar\imath\mbox{ and }
\bar\jmath\bigr\}.
\]
\end{enumerate}
We shall call the family $\{\mathbf{X}^{\mathbf{n}}; n\le\lfloor
1/\ga\rfloor\}$ a rough path over $X$ (it is also referred to as the
truncated signature of $X$ in
\cite{FV-bk}).

Once a rough path over $X$ is defined, the theory described in \cite
{FV-bk,Gu,LQ-bk} can be seen as a procedure which allows
us to construct, starting from the family $\{\mathbf{X}^{\mathbf{n}};
n\le\lfloor1/\ga\rfloor\}$, the complete stack $\{\mathbf
{X}^{\mathbf{n}}; n\ge1\}$. Furthermore, with the rough path over
$X$ in hand, one can also define rigorously and solve differential
equations driven by~$X$.

The above general framework leads thus naturally to the question of a
rough path construction for standard stochastic processes. The first
example one may have in mind concerning this issue is arguably the case
of a $d$-dimensional fractional Brownian motion (fBm) $B=(B(1),\ldots
,B(d))$ with Hurst parameter $H\in(0,1)$. This is a Gaussian process
with zero mean whose components are independent and with covariance
function given by
\[
\be(B_t(i)B_s(i) )= \tfrac12 (t^{2H} + s^{2H} - |t-s|^{2H}
),\qquad
s,t\in\R_+.
\]
For $H=\frac12$ this is just the usual Brownian motion. For any $H\in
(0,1)$, the variance of the increments of $B$ is then given by
\[
\be\bigl[\bigl(B_t(i)-B_s(i) \bigr)^2\bigr]= (t-s)^{2H},\qquad (s,t)\in\cs
_{2,T}, i=1,\ldots, d,
\]
and this implies that almost surely the trajectories of the fBm are
$\gamma$-H\"{o}lder continuous for any $\gamma<H$, which justifies the
fact that the fBm is the canonical
example for a rough path construction.

The first successful rough path analysis for $B$ has been implemented
in \cite{CQ} by means of a linearization of the fBm path, and it leads
to the construction of a family $\{\bb^{\mathbf{1}},\bb^{\mathbf{2}},\bb^{\mathbf{3}}\}$
satisfying (1), (\ref{eq:multiplicativity}) and (\ref
{eq:geom-rough-path}), for any $H>1/4$ (see also \cite{FV} for a
generalized framework). Some other constructions can be found in \cite
{FV,NNT,TT} by means of stochastic analysis methods, and in \cite{Un}
thanks to complex analysis tools. In all those cases, the barrier
$H>1/4$ remains, and it has long been believed that this was a natural
boundary, in terms of regularity, for an accurate rough path construction.

Let us describe now several recent attempts to go beyond the threshold
$H=1/4$. One should first quote the interesting paper \cite{LV}, where
a general construction of a rough path is performed by means of a
discretization procedure. However, the rough path constructed in this
reference is only defined on dyadic points, and then extended to any
real positive number by an abstract analytic result.
The complex analysis methods used in \cite{TU} also allowed
the authors to build a rough path above a process $\Gamma$ called
analytic fBm, which is a complex-valued process whose real and
imaginary parts are fBm, for any value of $H\in(0,1)$. It should be
mentioned, however, that $\Re\Gamma$ and $\Im\Gamma$ are not
independent, and thus the arguments in \cite{TU}
cannot be extrapolated to the real-valued fBm. 
Then a series of brilliant ideas developed in \cite{Un09a,Un09b} lead
to the rough path construction in the real-valued case.
We will try now to summarize briefly, in very vague terms, this series
of ideas (see Section \ref{sec:intg-order-2} for a more detailed
didactic explanation):

\mbox{}\hphantom{i}(i) Consider a smooth approximation $B^\ep$ of the fBm $B$ and
the corresponding approximation $\bb^{\bn,\ep}$ of $\bb^{\bn}$. Clearly
$\bb^{\bn,\ep}$ satisfies relation (\ref{eq:multiplicativity}), but may
diverge as $\ep\to0$ whenever $H<1/4$. Then, one can decompose $\bb
^{\bn,\ep}_{st}$ as $\bb^{\bn,\ep}_{st}=\ba^{\bn,\ep}_{st}+\bc
^{\bn,\ep
}_{st}$, where $\bc^{\bn,\ep}$ is the increment of a function $f$,
namely $\bc^{\bn,\ep}_{st}=f_t-f_s$, and $\ba^{\bn,\ep}$ is
obtained as
a boundary term in the integrals defining $\bb^{\bn,\ep}$. As explained
in Section~\ref{sec:intg-order-2}, a typical example of such a
decomposition is given (for $n=2$) by $\ba_{st}^{\mathbf{2},\ep
}(i_1,i_2)=-B_{s}^{\ep}(i_1) \der B_{st}^{\ep}(i_2)$ and $\bc
_{st}^{\mathbf{2}
,\ep}(i_1,i_2)= \int_s^t B_{u}^{\ep}(i_1) \,dB_{u}^{\ep}(i_2)$, and in
this case $f_t(i_1,i_2)=\int_0^t B_{u}^{\ep}(i_1) \,dB_{u}^{\ep}(i_2)$.
Then it can be easily checked, thanks to the relation $\bc^{\bn,\ep
}_{st}=f_t-f_s$, that $\bc^{\bn,\ep}_{st}-\bc^{\bn,\ep}_{su}-\bc
^{\bn
,\ep}_{ut}=0$ for any
$(s,u,t)\in\cs_{3,T}$. This means that replacing $\bb_{st}^{\bn,\ep}$
by $\ba_{st}^{\bn,\ep}=\bb^{\bn,\ep}-\bc_{st}^{\bn,\ep}$ does not
affect the multiplicative property (\ref{eq:multiplicativity}) of $\bb
^{\bn,\ep}$.
On the other hand, the boundary term $\ba^{\bn,\ep}_{st}$ is usually
easily seen to be convergent as $\ep\to0$ to some limit $\ba^{\bn
}_{st}$. Then, the limit $\ba^{\bn}_{st}$ should fulfill the desired
multiplicative property, but it does not exhibit the desired H\"{o}lder
regularity $(kH)^-$. It should also be noticed that $\ba^{\bn,\ep
}_{st}$ is not the only function of two variables sharing the
multiplicative property with $\bb^{\bn,\ep}$. We refer to
Section \ref{sec:intg-order-2} for further details, but let us mention
that another
possibility for $n=2$ is the boundary term $\der X_{st}^{\ep}(i_1)
X_{t}^{\ep}(i_2)$, which is easily seen to satisfy relation (\ref
{eq:multiplicativity}).

(ii)
The essential point in Unterberger's method is then the following:
carry out the above program for some given regularizations of the fBm
path. Then, it turns out that there is a choice of boundary terms such
that their sum satisfies the desired H\"older and multiplicative
properties. This idea has been successfully implemented in \cite
{Un09a,Un09b}, providing an explicit construction of a rough path
associated to $B$. However, this construction is rather long and
intricate, because the changes in the order of integration in the
multiple integrals are coded by admissible cuts in some trees
associated to multiple integrals. This language, well known by
algebraists \mbox{\cite{CL,Fo}}, numerical analysts \cite{Bu,HLR} and
theoretical physicists \cite{CK}, may, however, sound difficult to the
noninitiated reader.

The purpose of the current paper is to take up the program initiated
in \cite{Un09a}, and construct a rough path over $B$ in a rather simple
way, using the stochastic integral representation of the fBm as a
Volterra Gaussian process.
We know that (see \cite{Nu06}, Proposition 5.1.3, for a justification)
for $H<1/2$, each component $B(i)$ of $B$ can be written as
%
%
\begin{equation}\label{eq:fbm-volterra}
B_t(i)=\int_{\R} K(t,u) \,dW_u(i),\qquad t\ge0,
\end{equation}
where $W=(W(1),\ldots,W(d))$ is a $d$-dimensional Wiener process, and
where the Volterra-type kernel $K$ is defined on $\R_+\times\R_+$ by
%
%
\begin{eqnarray}\label{eq:def-K}\qquad
K(t,u)&=& c_H \biggl[ \biggl(\frac{u}{t}\biggr)^{1/2-H} (t-u)^{H-1/2}
\nonumber\\[-8pt]\\[-8pt]
&&\hspace*{17.6pt}{}+\biggl(\frac12-H\biggr) u^{1/2-H} \int_u^t v^{H-3/2}
(v-u)^{H-1/2} \,dv \biggr] \mathbf{1}_{\{0<u<t\}},
\nonumber
\end{eqnarray}
with a strictly positive constant $c_H$, whose exact value is
irrelevant for our purposes.
Then we show that the simple trick described at point (ii) above can be
applied in a straightforward way using the Volterra representation,
leading to a simple general formula for the multiple integrals $\bb
^{\bn}$.
To be more specific, let us describe the main result of this paper.
\begin{theorem}\label{thm:main-thm}
Let $B$ be a $d$-dimensional fractional Brownian motion with Hurst
parameter $H\in(0,1/2)$, admitting representation (\ref
{eq:fbm-volterra}). For $2\le n\le\lfloor1/H\rfloor$, any tuple
$(i_1,\ldots,i_n)$ of elements of $\{1,\ldots,d\}$, $1\le j\le n$ and
$(s,t)\in\cs_{2,T}$, set
%
%
\begin{eqnarray}\label{eq:def-hat-Bn}
&&\hbb_{st}^{\bn,j}(i_1,\ldots,i_n) \nonumber\\
&&\qquad=
(-1)^{j-1} \int_{A_{j}^{n}} \prod_{l=1}^{j-1} K(s,u_l) [
K(t,u_j)-K(s,u_j)]\\
&&\qquad\hspace*{82.7pt}{}\times
\prod_{l=j+1}^{n} K(t,u_l)\,dW_{u_1}(i_1) \cdots dW_{u_n}(i_n),
\nonumber
\end{eqnarray}
where the kernel $K$ is given by (\ref{eq:def-K}) and $A_{j}^{n}$ is
the subset of $[0,t]^n$ defined by
\begin{eqnarray*}
A_{j}^{n}&=&\{(u_1,\ldots,u_n)\in[0,t]^{n};\\
&&\hspace*{5.7pt} u_j=\min(u_1,\ldots
,u_n),
u_1>\cdots> u_{j-1}
\mbox{ and } u_{j+1}<\cdots<u_n\}.
\end{eqnarray*}
Notice that the
multiple stochastic integral in (\ref{eq:def-hat-Bn}) is understood in
the Strato\-novich sense, and is well defined as a $L^2(\Omega)$ random
variable as long as $n\le\lfloor1/H\rfloor$. Set also $\bb^{\mathbf{1}
}_{st}(i)=B_t(i)-B_s(i)$, and for $2\le n\le\lfloor1/H\rfloor$,
%
%
\begin{equation}\label{eq:def-Bn-intro}
\bb_{st}^{\bn}(i_1,\ldots,i_n)=\sum_{j=1}^{n } \hbb_{st}^{\bn
,j}(i_1,\ldots,i_n).
\end{equation}
Then the family $\{\bb^{\bn}; 1\le n\le\lfloor1/H\rfloor\}$ defines
a rough path over $B$, in the sense that $\bb^\bn$ is almost surely
$n\ga$-H\"{o}lder continuous for any $\ga<H$, and that it satisfies
relations (\ref{eq:multiplicativity}) and (\ref{eq:geom-rough-path}).
\end{theorem}

As announced above, formula (\ref{eq:def-hat-Bn}) defines in a compact
and simple way the (substitute to) iterated integrals of $B$ with
respect to itself. Furthermore, this formula also yields a reasonably
short way to estimate the moments of $\bb_{st}^{\bn}$, and thus its
H\"
older regularity. It should be mentioned, however, that our
construction is not as general as the one proposed in \cite{Un09b},
though it can be extended to a broad class of Gaussian Volterra
processes. More precisely, the reader can check that the only
properties of the kernel $K$ used in this paper are
\[
|K(t,u)| \le C[ (t-u)^{H-1/2} + u^{H-1/2}]\quad \mbox
{and}\quad
|\partial_t K(t,u)| \le C (t-u)^{H-3/2}
\]
for any $H\in(0,1/2)$ and for some constant $C>0$. It is also worth
mentioning at this point that our representation (\ref
{eq:def-Bn-intro}) of $\bb^{\bn}$ is adapted to the past of the path $B$.

It is hard to compare our main result with the one given in \cite{LV},
due to the abstract nature of the latter. We can, however, say a few
words about the relationship between the processes $\bb^\bn$ we have
produced and the pathwise ones constructed in the aforementioned
references \cite{CQ,FV,NNT,Un}, as well as with the recent objects
introduced in \cite{Un09b}.

\mbox{}\hphantom{i}(i)
When $1/4<H\le1/2$, let us denote by $\bb^{\mathbf{2},\mathrm{p}}$ (where
$\mathrm{p}$
stands for pathwise) the double iterated integral constructed in
\cite{CQ,FV,NNT,Un}. Notice that these integrals all coincide as limit
of Riemann sums (a fact which is mentioned in \cite{NTU}). One has then
to distinguish two situations:

(1) For $H=1/2$, a slight extension of our construction also
allows to define $\bb^\mathbf{2}$ for Brownian motion, and it is readily
checked in this case that $\bb^\mathbf{2}$ coincides with the usual
Stratonovich double iterated integral.

(2) When $1/4<H< 1/2$, we know that $\der\bb^{\bn}=\der
\bb^{\bn
,\mathrm{p}}$, and it can be seen from this relation that $\bb^{\bn}$ and
$\bb^{\bn,\mathrm{p}}$ only differ by the increment of a function~$f$.
This nontrivial correction term is identified at Section \ref
{sec:relation-other}. Notice that the correction term for $\bb^{\mathbf{3}}$
could be identified as well, but we did not include these computations
for the sake of conciseness.

(ii)
For $H\le1/4$, we shall see that our iterated integrals can be
considered under the framework of the rough path constructions by
Fourier normal ordering contained in \cite{Un09b}. As mentioned above,
our main result gives a more direct an elementary (though less general)
representation of the iterated integrals. This representation only uses
direct (as opposed to Fourier) coordinates and is adapted with respect
to the underlying fBm $B$.
All these considerations will be developed at Section \ref{sec:relation-other}.

Here is how our article is divided: some preliminary results, including
algebraic integration vocabulary, some estimates on the kernel $K$ and
It\^{o}--Stratonovich corrections, are given in Section \ref
{sec:preliminaries}. Then the basic ideas of the construction are
implemented in Section \ref{sec:intg-order-2} on second order iterated
integrals. This section is thus intended as a didactic introduction to
the construction, and could be enough for a first quick glimpse at the
topic. Then we give all the details concerning the general iterated
integral definition and prove Theorem \ref{thm:main-thm} in
Section \ref{sec:general-case}. Finally, Section~\ref{sec:relation-other}
establishes some links between our integrals and other well established
iterated integrals for fBm.

\section{Preliminaries}
\label{sec:preliminaries}
This section is first devoted to recall some notational conventions for
a special subset (called set of increments) of functions of several
variables. These conventions are taken from the algebraic integration
theory as explained in \cite{Gu,GT}. We will then recall some basic
estimates on iterated Stratonovich integrals with respect to the Wiener
process, which turn out to be useful for the remainder of the article.

\subsection{Some algebraic integration vocabulary}\label
{sec:algebraic-vocabulary}
The current section is not intended as an introduction to algebraic
integration, which would be useless for our purposes. However, we shall
use in the sequel some
notation taken from this method of rough paths analysis, and we shall
proceed to recall them now.

The algebraic integration setting is based on the notion of increment,
together with an
elementary operator $\der$ acting on them. The notion of increment can
be introduced in the following way: for an arbitrary real number
$T> 0$, a vector space~$V$, and an integer $k\ge1$, we denote by $\cs
_{k,T}$ the $k$th order simplex on $[0,T]$, and by
$\cac_k(V)$ the set of continuous functions $g \dvtx\cs_{k,T} \to V$ such
that $g_{t_1 \cdots t_{k}} = 0$
whenever $t_i = t_{i+1}$ for some $i\le k-1$.
Such a function will be called a
\textit{$(k-1)$-increment}, and we will
set $\cac_*(V)=\bigcup_{k\ge1}\cac_k(V)$. The operator $\der$
alluded to above can be seen as an operator acting on
$k$-increments,
and is defined as follows on $\cac_k(V)$:
%
%
\begin{equation}
\label{eq:coboundary}
\delta\dvtx\cac_k(V) \to\cac_{k+1}(V),\qquad
(\delta g)_{t_1 \cdots t_{k+1}} = \sum_{i=1}^{k+1} (-1)^i g_{t_1
\cdots\hat t_i \cdots t_{k+1}} ,
\end{equation}
where $\hat t_i$ means that this particular argument is omitted.
Then a fundamental property of $\der$, which is easily verified,
is that
$\delta\delta= 0$, where $\delta\delta$ is considered as an operator
from $\cac_k(V)$ to $\cac_{k+2}(V)$.
We will denote $\cZ\cac_k(V) = \cac_k(V) \cap\operatorname
{Ker}\delta$.

Some simple examples of actions of $\der$,
which will be the ones we will really use throughout the paper,
are obtained by letting
$g\in\cac_1$ and $h\in\cac_2$. Then, for any $(s,u,t)\in\cs
_{3,T}$, we have
%
%
\begin{equation}
\label{eq:simple_application}
(\der g)_{st} = g_t - g_s\quad
\mbox{and}\quad
(\der h)_{sut} = h_{st}-h_{su}-h_{ut},
\end{equation}
and in this particular case, it can be trivially checked that for any
$g\in\cac_1$, one has $\der\der g=0$. Conversely, any $h\in
{\mathcal Z}\cac_2$
can be written as $h=\der g$ for an element $g\in\cac_1$. In the sequel
of the paper, we shall write for two elements $h^1,h^2\in\cac_2$
%
%
\begin{equation}\label{eq:convention-ZC2}
h^1\stackrel{{\mathcal Z}\cac_2}{=}h^2
\quad\mbox{iff}\quad
h^1=h^2+z
\qquad\mbox{with }
z\in{\mathcal Z}\cac_2.
\end{equation}
Otherwise stated, $h^1\stackrel{{\mathcal Z}\cac_2}{=}h^2$ iff $\der
h^1=\der h^2$.

Notice that our future discussions will rely on some
analytical assumptions made on elements of
$\cac_k(V)$. Suppose $V$ is equipped with a norm \mbox{$|\cdot|$}.
We measure the size of the increments by H\"older norms
defined in the following way: for $g \in\cac_2(V)$ let
%
%
\begin{equation}\label{eq:def-norm-C2}
\Vert g\Vert_{\mu} \equiv
\sup_{(s,t)\in\cs_{2,T}}\frac{|g_{st}|}{|t-s|^\mu}\quad
\mbox{and}\quad
\cac_2^\mu(V)=\{g \in\cac_2(V); \Vert g\Vert_{\mu}<\infty\}.
\end{equation}
With this notation, we also set $\cac_1^\mu(V)=\{ f \in\cac_1(V);
\Vert\der f\Vert_\mu< \infty\}$ (notice that the sup norm of $f$ is not
taken into account in this definition). In the same way, for $h \in
\cac
_3(V)$, set
%
%
\begin{eqnarray}
\label{eq:norm-C22}
\Vert h\Vert_{\gamma,\rho} &=& \sup_{(s,u,t)\in\cs_{3,T}}
\frac{|h_{sut}|}{|u-s|^\gamma|t-u|^\rho},\\
\label{eq:norm-C3}
\Vert h\Vert_\mu&\equiv&
\inf\biggl\{\sum_i \Vert h_i\Vert_{\rho_i,\mu-\rho_i} ; h =\sum_i
h_i,
0 < \rho_i < \mu\biggr\},
\end{eqnarray}
where the last infimum is taken over all sequences $\{h_i \in\cac_3(V)
\}$ such that $h
= \sum_i h_i$ and for all choices of the numbers $\rho_i \in(0,\mu)$.
Then $\Vert\cdot\Vert_\mu$ is easily seen to be a norm on $\cac
_3(V)$, and
we set
\[
\cac_3^\mu(V):=\{h\in\cac_3(V); \Vert h\Vert_\mu<\infty\}.
\]
%
In order to avoid ambiguities, we shall denote by $\cn[f; \cac
_j^\mu
(V)]$ the $\mu$-H\"{o}lder norm (or semi-norm) on the space $\cac
_j(V)$, for $j=1,2,3$.

The lemma below, borrowed from \cite{Gu}, Lemma 4, will be an essential
tool for the analysis of H\"{o}lder-type regularity of our increments:
\begin{lemma}\label{GRR-2}Let $\kappa>0$ and $p \geq1$. Let $R \in
\cac_2(\R^l)$, with $\delta R \in\cac_3^{\ka}(\R^l)$ in the sense
given by (\ref{eq:norm-C3}).
If
\[
\int_{\cs_{2,T}} \frac{|R_{uv}|^{2p}}{|u-v|^{2 \kappa p +4}} \,du \,dv
< \infty,
\]
then $R \in\cac_2^{\kappa}(\R^l)$. In particular, there exists a
constant $C_{\kappa, p,l}>0$, such that
\begin{eqnarray*}
\cn[ R; \cac_2^{\kappa}(\R^l) ] & \leq&
C_{\kappa
,p,l} \biggl(\int_{\cs_{2,T}} \frac{| R_{uv} |^{2p}}{|u-v|^{2 \kappa
p+4}} \,du \,dv \biggr)^
{1/({2p})}\\
&&{} + C_{\kappa,p,l} \cn[ \delta R; \cac_3^{\ka
}( \R
^l) ].
\end{eqnarray*}
\end{lemma}

\subsection{Analytic bounds on the fractional Brownian kernel}
We gather in this section some technical bounds on the kernel $K$
involved in the Volterra representation of $B$, for which we use the
following convention (valid until the end of the article): for two
positive quantities $a$ and $b$, we write $a\lesssim b$ whenever there
exists a universal constant $C$ such that $a \le C b$.

First, a classical bound on $K$ is the following:
\begin{lemma}
Let $K$ be the fBm kernel defined by (\ref{eq:def-K}). Then for any $0<
u < t$, one has
%
%
\begin{equation}\label{eq:bnd-kernel}\qquad
|K(t,u)|\lesssim(t-u)^{H-1/2} + u^{H-1/2}
\quad\mbox{and}\quad
|\partial_t K(t,u)|\lesssim(t-u)^{H-3/2}.
\end{equation}
\end{lemma}

The following simple integral estimate on $K$ also turns out to be useful:
\begin{lemma}\label{lem:bnd-K-1}
Let $0<v<t\le T$. Then
$
\int_v^t K^2(t,w) \,dw \lesssim(t-v)^{2H}.
$
\end{lemma}
\begin{pf}
Invoking the bound (\ref{eq:bnd-kernel}) on $K$, we have
\begin{eqnarray*}
\int_v^t K^2(t,w) \,dw &\lesssim&
\int_{v}^{t} [(t-w)^{H-1/2} + w^{H-1/2} ]^{2} \,dw \\
&\lesssim&\int_{v}^{t} (t-w)^{2H-1} \,dw
+\int_{v}^{t} w^{2H-1} \,dw\\
&\lesssim&(t-v)^{2H} + (t^{2H}-v^{2H}).
\end{eqnarray*}
Furthermore, since $a^\al-b^\al\le(a-b)^{\al}$ for any $0\le b<a$ and
$\al\in(0,1)$, we end up with $\int_v^t K^2(t,w) \,dw \lesssim
(t-v)^{2H}$, which is our claim.
\end{pf}

We shall also use a slightly more elaborated result on $K$:
\begin{lemma}\label{lem:bnd-K-2}
Let $0<s<t\le T$, assume $H<1/2$ and consider the quantity
\[
I_{st}=\int_0^t [K(t,u_1)- K(s,u_1) ]^2 \biggl(\int_{u_1}^{t}
K^2(t,u_2) \,du_2 \biggr)\,du_1,
\]
where we recall that we have used the convention $K(t,u)=K(t,u)\mathbf{1}_{[0,t)}(u)$.
Then $|I_{st}|\lesssim|t-s|^{4H}$.
\end{lemma}
\begin{pf}
According to the fact that $K(t,u)=0$ whenever $u\ge t$, we obtain the
expression
\begin{eqnarray*}
I_{st}&=&\int_0^s [K(t,u_1)- K(s,u_1) ]^2 \biggl(\int_{u_1}^{t}
K^2(t,u_2) \,du_2 \biggr)\,du_1 \\
&&{} +\int_s^t K^2(t,u_1) \biggl(\int_{u_1}^{t} K^2(t,u_2) \,du_2 \biggr)\,du_1
:= I_{st}^{1}+I_{st}^{2}.
\end{eqnarray*}
Let us bound now the first of those terms: thanks to Lemma \ref
{lem:bnd-K-1}, one can write
$\int_{u_1}^{t} K^2(t,u_2) \,du_2\lesssim(t-u_1)^{2H}$. Moreover, for
$0\le u <s$ the bound (\ref{eq:bnd-kernel}) on $\partial_t K(t,u)$ yields
%
%
\begin{equation}\label{eq:bnd-delta-K}
|K(t,u) -K(s,u) |= \biggl|\int_s^t \partial_v K(v,u)\, dv \biggr|
\lesssim(s-u)^{H-1/2}- (t-u)^{H-1/2},\hspace*{-28pt}
\end{equation}
and thus, putting these two estimates together, we obtain
\[
I_{st}^{1} \lesssim
\int_0^s [(s-u)^{H-1/2}- (t-u)^{H-1/2} ]^2
(t-u)^{2H} \,du.
\]
Performing the changes of variable $v=s-u$ and $y=v/(t-s)$, we end up with
\[
I_{st}^{1} \lesssim
(t-s)^{4H} \int_0^{s/(t-s)} [(1+y)^{H-1/2}- y^{H-1/2} ]
^{2} (1+y)^{2H} \,dy.
\]
Furthermore, it is easily checked that $\int_0^{\infty} [
(1+y)^{H-1/2}- y^{H-1/2} ]^{2} (1+y)^{2H} \,dy$ is a
convergent integral whenever $H<1/2$, which gives the desired bound for
$I_{st}^{1}$. The term $I_{st}^{2}$ is in fact easier to handle, and we
leave those details to the reader for
the sake of conciseness. Then, the estimates on $I_{st}^{1}$ and $I_{st}^{2}$
yield our claim.
\end{pf}

Finally, the following related integral bound also turns out to be an
important estimate for the analysis of $n$th order iterated integrals:
\begin{lemma} \label{lem2}
Suppose that $2kH<1$. For $A>0$, set
\[
\beta_A= \int_{0}^{A }[ y^{H-1/2}-(1+y)^{H-1/2}]
[ y^{H-1/2}+(A-y)^{H-1/2}]
y^{2(k-1)H}\,dy.
\]
Then $\sup_{A>0} \beta_A <\infty$.
\end{lemma}
\begin{pf}
We can write $\beta_A\le\al_A+\ga_A$, with
\begin{eqnarray*} \label{e3}
\al_A &=& \int_0^\infty[ y^{H-1/2}-(1+y)^{H- 1/2}]
y^{2(k-1)H+H -1/2} \,dy, \\
\gamma_A&=& \int_{0}^{A }[ y^{H-1/2}-(1+y)^{H- 1/2}
] (A-y)^{H-1/2}
y^{2(k-1)H} \,dy.
\end{eqnarray*}
One can check easily, as in the proof of Lemma \ref{lem:bnd-K-2}, that
$\al_A$ is finite as long as $2kH<1$. On the other hand, an obvious
change of variables yields
%
%
\begin{equation}\label{eq:gamma-A-chg-vb}
\gamma_A=A^{2kH}\int_{0}^{1 } h_A(y) (1-y)^{H-1/2}
y^{2(k-1)H}\,dy,
\end{equation}
where the (positive) function
$h_A$ is defined on $\R_+$ by $h_A(y)=y^{H-1/2}-(\frac
1A+y)^{H-1/2}$. We now use two elementary estimates
\[
h_A(y)\le\biggl( \frac12 -H\biggr) \frac{y^{H-3/2}}{A}
\quad\mbox{and}\quad
h_A(y)\le y^{H-1/2},
\]
and we obtain
%
\begin{eqnarray*}
h_A(y) &=& h_A(y)^{2kH} h_A(y)^{1-2kH} \\
&\le&\biggl( \biggl( \frac12 -H\biggr) \frac1A y^{H-3/2} \biggr)
^{2kH} y^{(H-1/2) (1-2kH)}\\
&=& \frac{c_{H,k} y^{(1-2k)H-1/2}}{A^{2kH}},
\end{eqnarray*}
%
where $c_{H,k}=(\frac12-H)^{2kH}$. Plugging this bound into (\ref
{eq:gamma-A-chg-vb}), we get
\[
\gamma_A \le c_{H,k} \int_{0}^{1 } (1-y)^{H-1/2} y^{-H-1/2} \,dy.
\]
This last integral being finite, our claim is now proved.
\end{pf}

\subsection{Contraction of Stratonovich iterated integrals}
An important tool in our analysis of iterated integrals will be a
general formula of It\^{o}--Stratonovich corrections for iterated
integrals. This kind of result has already been obtained in the
literature, and for our purposes, it will be enough to use a particular
case of \cite{BA}, Proposition 1, recalled here for further use. Note
that we need an additional notation for this intermediate result: we
set $dY$ for the Stratonovich-type differential with respect to a
process $Y$, while the It\^{o}-type differential is denoted by
$\partial Y$.
\begin{proposition}\label{prop:ito-strato}
Let $Y=(Y(1),\ldots,Y(n))$ be a $n$-dimensional martingale of Gaussian
type, defined on an interval
$[s,t]$, of the form $Y_u(j)=\int_s^u \psi_v(j) \,dW_v(i_j)$ for a
family of $L^2([s,t])$ functions $(\psi(1),\ldots,\psi(n))$, a set of
indices $(i_1,\ldots,i_n)$ belonging to $\{1,\ldots,d\}^n$ and where we
recall that $(W(1),\ldots,$ $W(d))$ is a $d$-dimensional Wiener process.
Then the following decomposition holds true:
\[
\int_{s\le u_1<\cdots<u_n\le t} dY_{u_1}(i_1)\cdots dY_{u_n}(i_n)
=\sum_{k=\lfloor n/2 \rfloor}^{n}\frac{1}{2^{n-k}} \sum_{\nu\in
D_{n}^{k}} J_{st}(\nu).
\]
%
In the above formula, the sets $D_{n}^{k}$ are subsets of $\{1,2\}^k$
given by
\[
D_{n}^{k}=\Biggl\{\nu=(n_1,\ldots,n_k); \sum_{j=1}^{k}n_j=n \Biggr\},
\]
and the
It\^{o}-type multiple integrals $J_{st}(\nu)$ are defined as follows:
\[
J_{st}(\nu)=
\int_{s\le u_1<\cdots<u_k\le t} \partial Z_{u_1}(1)\cdots\partial Z_{u_k}(k),
\]
%
where, setting $\sum_{l=1}^{j}n_l=m(j)$, we have
\[
Z(j)=Y\bigl(i_{m(j)}\bigr) \qquad\mbox{if } n_j=1,
\]
and
\[
Z_u(j)=\biggl(\int_s^u \psi_v\bigl(m(j)-1\bigr) \psi_v(m(j)) \,dv \biggr)\mathbf{1}
_{(i_{m(j)-1}=i_{m(j)})}\qquad
\mbox{if } n_j=2.
\]
\end{proposition}

The previous It\^{o}--Stratonovich decomposition allows
us to bound the
second order moment of iterated Stratonovich integrals in the following way:
\begin{lemma}
\label{lem1}Let $\varphi\in L^{2}([s,t])$. Consider the Stratonovich
iterated integral%
\[
I_{st}^{n}(\varphi)=\int_{s<u_{1}<\cdots<u_{n}<t} \prod
_{i=1}^{n}\varphi
(u_{i}) \,dW_{u_{1}}(i_{1})\cdots dW_{u_{n}}(i_{n}).
\]
Then
%
%
\begin{equation} \label{e1}
\be[ I_{st}^{n}(\varphi)^{2}] \leq C\biggl( \int
_{s}^{t}\varphi
(u)^{2}\,du\biggr) ^{n},
\end{equation}
where the constant $C$ depends on $n$ and the multiindex $(
i_{1},\ldots,i_{n}) $.
\end{lemma}
\begin{pf}
By Proposition \ref{prop:ito-strato}, we can decompose the Stratonovich
integral $%
I_{st}^{n}(\varphi)$ into a sum of It\^{o} integrals%
\[
I_{st}^{n}(\varphi)=\sum_{k=\lfloor n/2\rfloor
}^{n}\frac
{1}{%
2^{n-k}}\sum_{\nu\in D_{n}^{k}}J_{st}(\nu),
\]
and it suffices to consider each It\^{o} integral $J_{st}(\nu)$.
Then we
proceed by recurrence with respect to $k$, with the notation of Proposition
\ref{prop:ito-strato}. Suppose first that $n_{k}=1$. Then,%
\[
J_{st}(\nu)=\int_{s}^{t}J_{su}(\nu^{\prime})\varphi(u)\,\partial
_{u}W(i_{n}),
\]
where $\nu^{\prime}=(n_{1},\ldots,n_{k-1})$. As a consequence,%
\[
\be[J_{st}(\nu)^{2}]=\int_{s}^{t}\be[ J_{su}(\nu^{\prime
})^{2}]
\varphi(u)^{2}\,du\leq\sup_{s\leq u\leq t}\be[ J_{su}(\nu
^{\prime
})^{2}] \int_{s}^{t}\varphi(u)^{2}\,du.
\]
On the other hand, if $n_{k}=2$, then $J_{st}(\nu)=\int
_{s}^{t}J_{su}(\nu^{\prime\prime})\varphi(u)^{2}\,du$, with $\nu
^{\prime\prime}=(n_{1}$, $\ldots,n_{k-2})$, and again%
%
\[
\be[J_{st}(\nu)^2
]\leq\sup_{s\leq u\leq t}\be[ J_{su}(\nu^{\prime\prime
})^{2}] \biggl( \int_{s}^{t}\varphi(u)^{2}\,du\biggr) ^{2}.
\]
By recurrence we obtain (\ref{e1}), where $C=( \sum_{k=
\lfloor
n/2\rfloor}^{n}\frac{| D_{n}^{k}|
}{2^{n-k}}) ^{2}$.
\end{pf}

\section{Iterated integrals of order 2}
\label{sec:intg-order-2}

In this section, we will define the element $\bb^{\mathbf{2}}$ announced
in Theorem \ref{thm:main-thm}. The study of this particular case will
(hopefully) allow us to introduce many of the technical ingredients
needed for the general case in a didactic way.

\subsection{Heuristic considerations}\label{sec:heuristic-order-2}

Let us first specify what is meant by an iterated integral of order 2:
according to the definitions contained in the \hyperref
[intro]{Introduction}, we are
searching for a process $\{\bb^{\mathbf{2}}_{st}(i_1,i_2); (s,t)\in\cs
_{2,T}, 1\le i_1,i_2\le d\}$ satisfying:

\begin{longlist}
\item the regularity condition $\bb^{\mathbf{2}}\in\cac_2^{2\ga}(\R
^{d^2})$;

\item the multiplicative property
%
%
\begin{equation}\label{eq:multiplicative-B2}\quad
\der\bb^{\mathbf{2}}_{sut}(i_1,i_2)= \bb^{\mathbf{1}}_{su}(i_1) \bb^{\mathbf{1}}_{ut}(i_2)
=[B_u(i_1)-B_s(i_1) ][B_t(i_2)-B_u(i_2) ],
\end{equation}
which should be satisfied almost surely for all
$(s,u,t)\in\cs_{3,T} $ and $1\le i_1,i_2\le d$;

\item the geometric relation, which can be read here as:
%
%
\begin{eqnarray}\label{eq:geometric-B2}
\bb^{\mathbf{2}}_{st}(i_1,i_2)+\bb^{\mathbf{2}}_{st}(i_2,i_1)=
\bb^{\mathbf{1}}_{st}(i_1) \bb^{\mathbf{1}}_{st}(i_2),\nonumber\\[-8pt]\\[-8pt]
&&\eqntext{(s,t)\in\cs_{2,T}, 1\le i_1,i_2\le d.}
\end{eqnarray}

In order to construct this kind of element, let us start with some
heuristic considerations, similar to the starting point of \cite
{Un09a}: assume for the moment that $X$ is a smooth $d$-dimensional
function defined on $\ott$. Then the natural notion of iterated
integral of order 2 for $X$ is obviously an element $\hat\bx^{\mathbf{2}}$,
defined in the Riemann sense by
%
%
\begin{eqnarray}\label{eq:intg-2-X}
\hat\bx^{\mathbf{2}}_{st}(i_1,i_2)&=&\int_{s\le u_1\le u_2\le t} dX_{u_1}(i_1)
\,dX_{u_2}(i_2)\nonumber\\[-8pt]\\[-8pt]
&=&\int_s^t [X_{u}(i_1)-X_{s}(i_1) ]\,dX_{u}(i_2).\nonumber
\end{eqnarray}
We shall now decompose $\hat\bx^{\mathbf{2}}$ into terms of the form $\ba^{\mathbf{2}}$
and $\bc^{\mathbf{2}}$ as explained in the \hyperref[intro]{Introduction}. In
our case, this can
be done in two ways: first, equation (\ref{eq:intg-2-X}) immediately yields
\[
\hat\bx^{\mathbf{2}}_{st}(i_1,i_2) = \hat\ba^{\mathbf{2},2}_{st} + \hat\bc^{\mathbf{2},2}_{st}
\]
with
\[
\hat\ba^{\mathbf{2},2}_{st}= -X_{s}(i_1) \der X_{st}(i_2),\qquad
\hat\bc^{\mathbf{2},2}_{st}= \int_s^t X_{u}(i_1)\, dX_{u}(i_2),
\]
where we have called those quantities $\hat\ba^{\mathbf{2},2}$ and $\hat\bc
^{\mathbf{2}
,2}$ because they involve increments of the second component $X(i_2)$
of $X$.
Notice now that $\hat\bc^{\mathbf{2},2}$ is the increment of a function $f$
defined as
$f_t=\int_0^t X_{u}(i_1) \,dX_{u}(i_2)$. Hence, according to
convention (\ref{eq:convention-ZC2}), one can write $\hat\bx^{\mathbf{2}
}(i_1,i_2)\zc\hat\ba^{\mathbf{2},2}$. By inverting the order of integration in
$u_1,u_2$ thanks to Fubini's theorem, we also obtain
\[
\hat\bx^{\mathbf{2}}_{st}(i_1,i_2) = \hat\ba^{\mathbf{2},1}_{st} + \hat\bc^{\mathbf{2},1}_{st}
\]
with
\[
\hat\ba^{\mathbf{2},1}_{st}= \der X_{st}(i_1) X_{t}(i_2) ,\qquad
\hat\bc^{\mathbf{2},1}_{st}= -\int_s^t X_{u}(i_2) \,dX_{u}(i_1),
\]
and thus $\hat\bx^{\mathbf{2}}(i_1,i_2)\zc\hat\ba^{\mathbf{2},1}$.

Let us go back now to the case of the $d$-dimensional fBm $B$. If we
wish the iterated integral $\bb^\mathbf{2}$ we are constructing to behave in a
similar manner as a Riemann-type integral, then, by the Chen property,
one should have $\delta\bb^{\mathbf{2}}= \delta\ba^{\mathbf{2},i}$, for $i=1,2$,
that is,
\[
\bb^{\mathbf{2}}(i_1,i_2)\zc\ba^{\mathbf{2},2}\quad \mbox{and}\quad
\bb^{\mathbf{2}}(i_1,i_2)\zc\ba^{\mathbf{2},1},
\]
with $\ba^{\mathbf{2},2}_{st}=-B_{s}(i_1) \der B_{st}(i_2)$ and $\ba^{\mathbf{2}
,1}_{st}= \der B_{st}(i_1) B_{t}(i_2)$. This means in particular,
according to the fact that $\der|_{{\mathcal Z}\cac_2}=0$, that both
$\ba^{\mathbf{2}
,1}$ and $\ba^{\mathbf{2},2}$ satisfy the multiplicative relation (\ref
{eq:multiplicative-B2}), as
it can be easily checked by direct computations. However, this naive
decomposition has an important drawback: the increments $\ba^{\mathbf{2},1}$
and $\ba^{\mathbf{2},2}$ only belong to $\cac_2^{\ga}$, instead of $\cac
_2^{2\ga
}$, for any $\ga<H$ (this point was also stressed in \cite{Un09a}).

Our construction diverges from \cite{Un09a} in the way we cope with the
regularity problem mentioned above. Indeed, we start from the following
observation: invoking the representation (\ref{eq:fbm-volterra}) of
$B$, one can write
\begin{eqnarray*}
\ba^{\mathbf{2},2}_{st}&=& -B_{s}(i_1) \der B_{st}(i_2)\\
&=&-\int_{\R} K(s,u_1) \,dW_{u_1}(i_1)
\int_{\R} [K(t,u_2)-K(s,u_2) ]\,dW_{u_2}(i_2) \\
&=&- \int_{\R^2} K(s,u_1) [K(t,u_2)-K(s,u_2) ]\,dW_{u_1}(i_1)
\,dW_{u_2}(i_2),
\end{eqnarray*}
where we recall that the stochastic differentials $dW$ are defined in
the Stra\-to\-no\-vich sense. In the same way, we get
\[
\ba^{\mathbf{2},1}_{st}=\int_{\R^2} [K(t,u_1)-K(s,u_1) ]K(t,u_2)
\,dW_{u_1}(i_1) \,dW_{u_2}(i_2).
\]
The idea in order to transform $\ba^{\mathbf{2},1},\ba^{\mathbf{2},2}$ into $\cac
_2^{2\ga}$ increments is then to replace the integrals over $\R^2$
above by integrals on the simplex, as mentioned in the \hyperref
[intro]{Introduction}.
Namely, we set now
%
%
\begin{eqnarray}\hspace*{28pt}
\label{eq:def-B2-1}
\hat\bb^{\mathbf{2},1}_{st}(i_1,i_2)&=&
\int_{u_1<u_2} [K(t,u_1)-K(s,u_1) ]K(t,u_2) \,dW_{u_1}(i_1)
\,dW_{u_2}(i_2),
\\
\label{eq:def-B2-2}
\hat\bb^{\mathbf{2},2}_{st}(i_1,i_2)&=&
- \int_{u_2<u_1} K(s,u_1) [K(t,u_2)-K(s,u_2) ]\,dW_{u_1}(i_1)
\,dW_{u_2}(i_2),
\end{eqnarray}
and notice that these formulas are a particular case of (\ref
{eq:def-hat-Bn}) for $n=2$. We shall see that $\hat\bb^{\mathbf{2},1}(i_1,i_2)$
and $\hat\bb^{\mathbf{2},2}(i_1,i_2)$ are elements of $\cac_2^{2\ga}$, but they
do not satisfy the multiplicative and geometric property anymore.
However, it is now easily conceived, by some symmetry arguments, that
the sum of these last two terms do satisfy the desired algebraic
properties again. Indeed, we set now
%
%
\begin{equation}\label{eq:def-B2}
\bb^{\mathbf{2}}_{st}(i_1,i_2)=\hat\bb^{\mathbf{2},1}_{st}(i_1,i_2)+\hat\bb^{\mathbf{2}
,2}_{st}(i_1,i_2),
\end{equation}
and we claim that $\bb^\mathbf{2}$ is a $\cac_2^{2\ga}(\R^{d^2})$ increment
which fulfills relations (\ref{eq:multiplicative-B2}) and (\ref
{eq:geometric-B2}). The remainder of this section is devoted to prove
these claims.
\end{longlist}

\subsection{Properties of the second order increment}\label
{sec:prop-second-intg}

It is obviously essential for the following developments to check that
$\bb^\mathbf{2}$ is a well defined object in $L^2(\Omega)$. The next
proposition asserts the existence of $\bb^2_{st}$ as a $L^2$ random
variable for all $s,t$ in the interval $\ott$.
\begin{proposition}\label{prop:bnd-second-moment-B2-st}
Let $H<1/2$, $(s,t)\in\cs_{2,T}$ and $\bb^\mathbf{2}_{st}$ be the matrix valued
random variable defined by (\ref{eq:def-B2}). Then
$\bb^\mathbf{2}_{st}(i_1,i_2)\in L^2(\Omega; \R^{d^2})$ and $\be[ |\bb^\mathbf{2}
_{st}|^2] \lesssim(t-s)^{4H}$.
\end{proposition}
\begin{pf}
Assume first $i_1\ne i_2$. We shall focus on the relation $\be[ ( \hat
\bb^{\mathbf{2},1}_{st}(i_1$, $i_2) )^2] \lesssim(t-s)^{4H}$, the bound on $\hat
\bb^{\mathbf{2},2}_{st}$ being obtained in a similar way. Now Stratonovich and
It\^{o}-type integrals coincide when $i_1\ne i_2$, and according to
expression (\ref{eq:def-B2-1}) we have
\begin{eqnarray*}
&&\be[(\hat\bb^{\mathbf{2},1}_{st}(i_1,i_2) )^2]\\
&&\qquad= \int_{u_1<u_2}
\bigl[K(t,u_1) \mathbf{1}_{[0,t]}(u_1)-K(s,u_1) \mathbf{1}_{[0,s]}(u_1) \bigr]^{2}\\
&&\qquad\hspace*{38pt}{}\times K^{2}(t,u_2) \mathbf{1}_{[0,t]}(u_2) \,du_1 \,du_2,
\end{eqnarray*}
which is exactly the quantity $I_{st}$ studied at Lemma \ref
{lem:bnd-K-2}. The desired bound
follows from Lemma \ref{lem:bnd-K-2}.

Let us now treat the case $i_1=i_2=i$, still concentrating our efforts
on the inequality $\be[ ( \hat\bb^{\mathbf{2},1}_{st}(i,i) )^2] \lesssim
(t-s)^{4H}$. In this context, Proposition \ref{prop:ito-strato} yields
the decomposition $\hat\bb^{\mathbf{2},1}_{st}(i,i)=M_{st}+V_{st}$, with
\begin{eqnarray*}
M_{st}&=&
\int_{u_1<u_2} [K(t,u_1)-K(s,u_1) ]K(t,u_2) \,\partial
W_{u_1}(i) \,\partial W_{u_2}(i), \\
V_{st}&=&\frac12
\int_0^t [K(t,u )-K(s,u ) ]K(t,u ) \,du,
\end{eqnarray*}
where we stress the fact that $V_{st}$ is a deterministic correction term.
It is thus obviously enough to obtain the bounds $\be
[M_{st}^2]\lesssim
(t-s)^{4H}$ and $V_{st}^2\lesssim(t-s)^{4H}$ separately, the first of
these bounds being obtained by evaluating $I_{st}$ in Lemma \ref
{lem:bnd-K-2} again. As far as $V_{st}$ is concerned, we make the decomposition
\[
V_{st} = \frac12 \int_0^s [K(t,u )-K(s,u ) ]K(t,u ) \,du
+ \int_s^t K(t,u)^2 \,du.
\]
The second term is bounded by a constant times $(t-s)^{2H}$ by Lemma
\ref{lem:bnd-K-1}. For the first term we use the estimate
\[
|K(t,u )-K(s,u )| \lesssim(t-s)^{2H} (s-u)^{-H -1/2},
\]
which trivially finishes the proof.
\end{pf}

By standard arguments (see \cite{TU}) it can be proved that the
estimates in Proposition \ref{prop:bnd-second-moment-B2-st} imply that
$\bb^{\mathbf{2}} \in C_2^{2H-} (\mathbb{R}^{d^2})$.

We are now equipped with the continuous version of $\bb^\mathbf{2}$ exhibited
in the last proposition, with which we will work without further
mention, and we are now ready to prove the algebraic relations
satisfied by our second order increment.
\begin{proposition}\label{prop:algebraic-B2}
The increment $\bb^\mathbf{2}$ defined by (\ref{eq:def-B2}) satisfies relations
(\ref{eq:multiplicative-B2}) and (\ref{eq:geometric-B2}).
\end{proposition}
\begin{pf}
Recall that we are now dealing with a continuous version of $\bb^\mathbf{2}$.
In fact, one can easily modify the arguments of \cite{TU} in order to
get a continuous version of the pair $(\bb^\mathbf{1},\bb^\mathbf{2})$. This means that
it is enough to check relations (\ref{eq:multiplicative-B2}) and (\ref
{eq:geometric-B2}) for some fixed $0\le s<u<t\le T$.

Let us then verify (\ref{eq:multiplicative-B2}) for $s,u,t\in\ott$ such
that $s<u<t$. It is readily seen, by writing the definitions of $\bb
^{\mathbf{2}
,1}_{st}(i_1,i_2), \bb^{\mathbf{2},1}_{su}(i_1,i_2)$ and $\bb^{\mathbf{2}
,1}_{ut}(i_1,i_2)$, that
\begin{eqnarray*}
\der\bb^{\mathbf{2},1}_{sut}(i_1,i_2)&=&
\int_{u_1<u_2} [K(u,u_1)-K(s,u_1) ]\\
&&\hspace*{26.6pt}{}\times
[K(t,u_2)-K(u,u_2)]\,dW_{u_1}(i_1) \,dW_{u_2}(i_2),
\end{eqnarray*}
the right-hand side of this equality being well defined as a $L^2$
random variable (a~fact which can be shown similarly to Proposition
\ref{prop:bnd-second-moment-B2-st}). Along the same lines, we also get
\begin{eqnarray*}
\der\bb^{\mathbf{2},2}_{sut}(i_1,i_2)&=&
\int_{u_1>u_2} [K(u,u_1)-K(s,u_1) ]\\
&&\hspace*{26.3pt}{}\times[K(t,u_2)-K(u,u_2)]\,dW_{u_1}(i_1) \,dW_{u_2}(i_2),
\end{eqnarray*}
and thus
\begin{eqnarray*}
\der\bb^{\mathbf{2}}_{sut}(i_1,i_2)&=&\der\bb^{\mathbf{2},1}_{sut}(i_1,i_2)+\der
\bb
^{\mathbf{2},2}_{sut}(i_1,i_2) \\
&=&\int_{\R^2} [K(u,u_1)-K(s,u_1) ]\\
&&\hspace*{13pt}{}\times
[K(t,u_2)-K(u,u_2)]\,dW_{u_1}(i_1) \,dW_{u_2}(i_2) \\
&=& \bb^\mathbf{1}_{su}(i_1) \bb^\mathbf{1}_{ut}(i_2),
\end{eqnarray*}
which is relation (\ref{eq:multiplicative-B2}).

As far as relation (\ref{eq:geometric-B2}) is concerned, reorder the
integration indices in (\ref{eq:def-B2-1}) in order to get
\[
\hat\bb^{\mathbf{2},1}_{st}(i_2,i_1)=
\int_{u_2<u_1} K(t,u_1) [K(t,u_2)-K(s,u_2) ]\,dW_{u_1}(i_1)
\,dW_{u_2}(i_2).
\]
Add this expression to (\ref{eq:def-B2-2}), which yields
%
%
\begin{eqnarray}\label{eq:geom-1}
&&\hat\bb^{\mathbf{2},1}_{st}(i_2,i_1)+\hat\bb^{\mathbf{2},2}_{st}(i_1,i_2)\nonumber\\
&&\qquad=\int_{u_2<u_1} [K(t,u_1)-K(s,u_2) ]\\
&&\qquad\hspace*{37.4pt}{}\times[K(t,u_2)-K(s,u_2) ]
\,dW_{u_1}(i_1) \,dW_{u_2}(i_2).
\nonumber
\end{eqnarray}
Exactly in the same way, we get
%
%
\begin{eqnarray}\label{eq:geom-2}
&&\hat\bb^{\mathbf{2},1}_{st}(i_1,i_2)+\hat\bb^{\mathbf{2},2}_{st}(i_2,i_1)\nonumber\\
&&\qquad=\int_{u_1<u_2} [K(t,u_1)-K(s,u_2) ]\\
&&\qquad\quad\hspace*{26.3pt}{}\times[K(t,u_2)-K(s,u_2) ]
\,dW_{u_1}(i_1) \,dW_{u_2}(i_2).
\nonumber
\end{eqnarray}
Putting together equations (\ref{eq:geom-1}) and (\ref{eq:geom-2}), our
claim (\ref{eq:geometric-B2}) is now readily checked.
\end{pf}

Finally, let us close this section by giving the proof of the announced
regularity result on $\bb^\mathbf{2}$.
\begin{proposition}\label{prop:regularity-B2}
The increment $\bb^\mathbf{2}$ is almost surely an element of $\cac_2^{2\ga
}(\R
^{d^2})$, for any $\ga<H$.
\end{proposition}
\begin{pf}
Consider a fixed H\"{o}lder exponent $\ga<H$.
The proof of this result is based on Lemma \ref{GRR-2},
which can be read here as $\cn[\bb^\mathbf{2};\cac_2^{2\ga}(\R
^{d^2})]\lesssim
A+D$, with
%
\[
A= \biggl(\int_{\cs_{2,T}} \frac{| \bb^{\mathbf{2}}_{uv} |^{2p}}{|u-v|^{4
\ga
p+4}} \,du \,dv \biggr)^{1/({2p})}
\quad\mbox{and}\quad
D= \cn[ \delta\bb^{\mathbf{2}}; \cac_3^{2\ga}( \R^{d^2}) ].
\]

Let us first deal with the term $D$ above: we have seen that $\bb^\mathbf{2}$
satisfies the multiplicative property (\ref{eq:multiplicative-B2}),
which can be summarized as $\delta\bb^{\mathbf{2}}=\der B\otimes\der B$.
Furthermore, $B\in\cac_1^\ga(\R^d)$ for any $\ga<H$, and thus, for any
$1\le i_1,i_2\le d$ and $0\le s <u <t \le T$
\[
|\delta\bb^{\mathbf{2}}_{sut}(i_1,i_2)|=|\der B_{su}(i_1)| |\der B_{ut}(i_2)|
\le\cn^{2}[B; \cac_1^\ga(\R^d)] |u-s|^{\ga} |t-u|^{\ga}.
\]
In other words, the quantity $\|\der\bb^{\mathbf{2}}\|_{\ga,\ga}$ defined by
(\ref{eq:norm-C22}) is almost surely finite, and according to
definition (\ref{eq:norm-C3}), we obtain that $D$ is also almost
surely finite.

We will now show that $A$ is finite almost surely when $p$ is large
enough, by proving that $\be[A]<\infty$. Indeed, invoking Jensen's
inequality we obtain
%
%
\begin{eqnarray}\label{eq:bnd-A}
\be[A ]&\le&
\biggl(\int_{\cs_{2,T}}\frac{\be[|\bb^{\mathbf{2}}_{uv}|^{2p}]}
{|u-v|^{4 \ga p+4}} \,du\,dv \biggr)^{1/({2p})}\nonumber\\[-8pt]\\[-8pt]
&\lesssim&
\biggl(\int_{\cs_{2,T}}\frac{\be^{p} [|\bb^{\mathbf{2}}_{uv}|^{2}]}
{|u-v|^{4 \ga p+4}} \,du\,dv \biggr)^{1/({2p})},\nonumber
\end{eqnarray}
where we have used the fact that $\bb^\mathbf{2}$ belongs to the second chaos
of $W$, on which all the $L^p$ norms are equivalent. On the other hand,
Proposition \ref{prop:bnd-second-moment-B2-st} gives $\be^{p} [
|\bb
^{\mathbf{2}}_{uv}|^{2}]\lesssim|u-v|^{4pH}$, and plugging this inequality
into (\ref{eq:bnd-A}), we obtain that $\be[A]$ is finite as long as
$p>1/(H-\ga)$.
\end{pf}

In conclusion, putting together the last two propositions, we have
constructed an element $\bb^\mathbf{2}$ which satisfies the properties
(i)--(iii) given at the beginning of the section, for any $H<1/2$.

\section{General case}
\label{sec:general-case}

The aim of this section is to prove Theorem \ref{thm:main-thm} in its
full generality. Recall that we define our substitute $\bb^{\bn}$ to
$n$th order integrals in the following way: for $2\le n\le\lfloor
1/H\rfloor$, any tuple $(i_1,\ldots,i_n)$ of elements of $\{1,\ldots
,d\}
$, $1\le j\le n$ and $(s,t)\in\cs_{2,T}$, set
%
%
\begin{eqnarray}\label{eq:def-hat-Bn-2}
&&\hbb_{st}^{\bn,j}(i_1,\ldots,i_n) \nonumber\\
&&\qquad=
(-1)^{j-1} \int_{A_{j}^{n}} \prod_{l=1}^{j-1} K(s,u_l) [
K(t,u_j)-K(s,u_j)]\\
&&\qquad\quad\hspace*{71.3pt}{}\times\prod_{l=j+1}^{n} K(t,u_l) \,dW_{u_1}(i_1) \cdots dW_{u_n}(i_n),
\nonumber
\end{eqnarray}
where the kernel $K$ is given by (\ref{eq:def-K}) and $A_{j}^{n}$ is
the subset of $[0,t]^n$ defined by
\begin{eqnarray*}
&&A_{j}^{n}=\{(u_1,\ldots,u_n)\in[0,t]^{n};\\
&&\hspace*{31.5pt} u_j=\min(u_1,\ldots
,u_n),
u_1>\cdots> u_{j-1}
\mbox{ and } u_{j+1}<\cdots<u_n\}.
\end{eqnarray*}
The 1-increment $\bb^{\bn}$ is then given by
%
%
\begin{equation}\label{eq:def-Bn-2}
\bb_{st}^{\bn}(i_1,\ldots,i_n)=\sum_{j=1}^{n-1} \hbb_{st}^{\bn
,j}(i_1,\ldots,i_n).
\end{equation}

It is obviously harder to reproduce the heuristic considerations
leading to this expression than in Section \ref{sec:heuristic-order-2}.
Let us just mention that the same kind of changes in the order of
integration allows us to produce some 1-increments similar to $\ba^{\mathbf{2}
,1},\ba^{\mathbf{2},2}$. Then the reordering trick yields some terms of the
form $\hbb_{st}^{\bn,j}$. After observing the form of several of these
terms, the general expression (\ref{eq:def-hat-Bn-2}) is then intuited
in a natural way.

\textit{Notation}: in order to write shorter formulas in the
computations below, we use the following conventions in the sequel,
whenever possible:

\begin{longlist}
\item
A product of kernels of the form $\prod_{j=1}^{n}K(\tau_j,u_j)$ will
simply be denoted by $\prod_{j=1}^{n}K_{\tau_j}$, meaning that the
variable $u_j$ has to be understood according to the position of the
kernel $K$ in the product.

\item
In the same context, we will also set $\der K_{st}$ for a quantity of
the form $K(t,u_j)-K(s,u_j)$.

\item
Furthermore, when all the $\tau_j$ are equal to the same instant $t$,
we write $\prod_{j=1}^{n}K(t,u_j)=K_t^{\otimes n}$.

\item
Finally, we will also shorten the notation for the increments of the
Wiener process $W$, and simply write $dW$ for $\prod_{j=1}^{n} dW_{u_j}(i_j)$.
\end{longlist}

All these conventions allow
us, for instance, to summarize formula (\ref{eq:def-hat-Bn-2}) into
%
%
\begin{equation}\label{eq:def-hat-Bn-2-summarized}
\hbb_{st}^{\bn,j}(i_1,\ldots,i_n) =
(-1)^{j-1} \int_{A_{j}^{n}} K_s^{\otimes(j-1)} \der K_{st}
K_t^{\otimes(n-j)} \,dW.
\end{equation}

\subsection{Moments of the $\mathbf{n}$th order integrals}

As in Section \ref{sec:prop-second-intg}, an important step of our
analysis is
a control of the second moment of $\bb^\bn$. This is given in the
following proposition.
\begin{proposition}\label{prop:moments-B-n}
For $n\leq\lfloor\frac{1}{H}\rfloor$, let $\mathbf{B}_{st}^{n}$ be
defined by (\ref{eq:def-Bn-2}). Then for $(s,t)\in\cs_{2,T}$, we have
\[
\be[ | \mathbf{B}_{st}^{n}| ^{2}] \leq C(t-s)^{2nH},
\]
for a strictly positive constant $C$.
\end{proposition}
\begin{pf}
Thanks to decomposition (\ref{eq:def-Bn-2}), it suffices to show that
for any fixed family of indexes $i_{1},\ldots
,i_{n}\in\{1,\ldots,d\}$ and for any $1\leq j\leq n-1$, we have%
\[
\be[ | \hat{\mathbf{B}}_{st}^{\bn,j}(i_{1},\ldots
,i_{n})|
^{2}] \leq C(t-s)^{2nH}.
\]
Invoking now expression (\ref{eq:def-hat-Bn-2}) for $\hbb^{\bn,j}$ and
decomposing the integral over the region $A_{j}$ appearing in the
definition of $\hat{\mathbf{B}}_{st}^{\bn,j}(i_{1},\ldots,i_{n})$ into
sums of
integrals over the simplex by means of Fubini's theorem, it suffices to
show an inequality of the type
%
%
\begin{equation}\label{eq:def-Q}
\be[(Q_{st})^2]\le C(t-s)^{2nH}
\qquad\mbox{with }
Q_{st}=
\int_{0<u_{1}<\cdots<u_{n}<t}
\der K_{st}
\prod_{i=2}^{n}K_{\tau_{i}} \,dW.\hspace*{-28pt}
\end{equation}
Notice that in the expression above, we made use of the
notation introduced at the beginning of the current section, and for
$i=1,\ldots,n$, we
assume $\tau_i=s$ or $t$. We concentrate our efforts now in proving
(\ref{eq:def-Q}).

Let us further decompose $Q$ into $Q=Q^{1}+Q^{2}$, where%
%
%
\begin{eqnarray}\label{eq:def-Q1-Q2}
Q^{1}_{st}&=&\int_{s<u_{1}<\cdots<u_{n}<t}
K_{t}^{\otimes n} \,dW\quad
\mbox{and}\nonumber\\[-8pt]\\[-8pt]
Q^{2}_{st}&=&\int_{0<u_{1}<\cdots<u_{n}<t,u_{1}<s}
\der K_{st}
\prod_{i=2}^{n}K_{\tau_{i}} \,dW ,\nonumber
\end{eqnarray}
as in the proof of Lemma \ref{lem:bnd-K-2}.
Notice that in $Q^{1}_{st}$ we have assumed that $\tau_{i}=t$ for
all~$i$, since otherwise this term vanishes. Moreover, the term
$Q^{1}_{st}$ can be handled using the
properties of the multiple Stratonovich integrals established in
Lem\-ma~\ref{lem1}, and applying the estimate obtained in Lemma \ref
{lem:bnd-K-1}. This yields easily the relation $\be
[(Q_{st}^{1})^2]\lesssim(t-s)^{2nH}$.

Concerning $Q^{2}_{st}$, one can write $Q^{2}_{st}=\sum_{j=1}^{n}B_{st}^{j}$
where%
\[
B_{st}^{j}=\int_{0<u_{1}<\cdots<u_{j}<s<u_{j+1}<\cdots<u_{n}<t}
\der K_{st}
\prod_{i=2}^{n}K_{\tau_{i}} \,dW .
\]
Notice that in the above equation $\tau_{i}=t$ if $i=j+1,\ldots,n$,
since we have again $B_{st}^{j}=0$ otherwise. Each term $B_{st}^{j}$
can thus be written as the
product of two factors: $B_{st}^{j}=C_{st}^{j}D_{st}^{j}$, where
for $j\ge2$
\[
C_{st}^{j}=\int_{0<u_{1}<\cdots<u_{j}<s}
\der K_{st}
\prod_{i=2}^{j}K_{\tau_{i}} \,dW
\]
and
\[
D_{st}^{j}=\int_{s<u_{j+1}<\cdots<u_{n}<t}
K_t^{\otimes(n-j)} \,dW,
\]
%
and for $j=1$, $C_{st}^{1}=\int_0^s \der K_{st} \,dW$ and $D_{st}^{1}$ is
given by the above formula.

The random\vspace*{1pt} variables $C_{st}^{j}$ and $D_{st}^{j}$ are independent, and
$\be[(D_{st}^{j})^{2}]$ can be bounded easily like $\be[(Q^1_{st})^2]$.
Hence we obtain
%
%
\begin{equation}\label{eq:bnd-B-st-j}
\be[(B_{st}^{j})^{2}]=\be[(C_{st}^{j})^{2}]\be[
(D_{st}^{j})^{2}]
\leq C\be[(C_{st}^{j})^{2}](t-s)^{2(n-j)H}.
\end{equation}

In order to bound the second moment of $C_{st}^{j}$, we express this
factor as a sum of It\^{o} integrals by means of Proposition \ref
{prop:ito-strato}. To do this, we give up for a moment our
convention on products of increments, and we define, for
$u\in[0,s]$ and $l=2,\ldots,j$, the processes
\[
Y_{u}(1)=\int_{0}^{u}[
K(t,v)-K(s,v)] \,dW_{v}(i_{1})
\quad\mbox{and}\quad
Y_{u}(l)=\int_{0}^{u}
K(\tau_{l},v)\,dW_{v}(i_{l}).
\]
%
Then, the processes $\{Y_{u}(l); 0\le u \le s\}$ are Gaussian
martingales and
\[
C_{st}^j= \int_{0<u_{1}<\cdots<u_{j}<s} dY_{u_1} (1) \,dY_{u_1}
(2)\cdots dY_{u_l} (l).
\]
Thus, a direct application of Proposition \ref{prop:ito-strato} yields
\[
C_{st}^{j}=\sum_{k=\lfloor j/2\rfloor}^{j}\frac
{1}{2^{j-k}}\sum_{\nu\in
D_{j}^{k}}J_{0s}(\nu),
\]
where
\[
J_{0s}(\nu)=\int_{0<u_{1}<\cdots<u_{k}<s}\partial Z_{u_{1}}(1)\cdots
\partial Z_{u_{k}}(k),
\]
%
for $\nu=(j_{1},\ldots,j_{k})$. Thus, setting $\sum
_{l=1}^{h}j_{l}=m(h)$, we have $Z(h)=Y(i_{m(h)})$ if $j_{h}=1$, and
$Z_{u}(h)=\langle Y(m(h)-1),Y(m(h))\rangle_{u}$ if
$j_{h}=2$ and
$i_{m(h)-1}=i_{m(h)}$, where $\langle\cdot,\cdot\rangle$ designates the
bracket of two continuous
martingales. We are going to
estimate $\be[J_{0s}(\nu)^{2}]$ using a
recursive argument.
This will be done in several steps.

\textit{Step} 1: \textit{suppose $j_{k},j_{k-1},\ldots,j_{1}=2$.} Then $j=2k$,
and we can assume that $i_m=i_{m-1}$ for $m=2,4,\ldots, 2k$, otherwise
$J_{0s}(\nu)=0$. The term $J_{0s}(\nu)$ is deterministic and it can
be expressed as follows:
\begin{eqnarray*}
J_{0s}(\nu)&=&
\int_{0<u_{1}<\cdots<u_{k}<s} [ K(t,u_{1})-K(s,u_{1})%
] K(\tau_{2},u_{1})\\
&&\hspace*{61pt}{}\times\prod_{h=2}^{k}K(\tau_{2h-1},u_{h})K(\tau_{2h},u_{h})
\,du_{1}\cdots du_{k}.
\end{eqnarray*}
As a consequence, owing to (\ref{eq:bnd-kernel}) and (\ref
{eq:bnd-delta-K}), we have
%
%
\begin{equation}\label{eq:bnd-J-0s-nu}
| J_{0s}(\nu)| \leq C\int_{0<u_{1}<\cdots<u_{k}<s}
\vp^{(1)}_{u_1} \prod_{h=2}^{k} \vp^{(2)}_{u_h} \,du_{1}\cdots du_{k},
\end{equation}
where
%
\[
\vp^{(1)}_{u_1} = [
(s-u_{1})^{H-{1/2}}-(t-u_1)^{H-{1/2}}] [
(s-u_{1})^{H-%
{1/2}}+u_{1}^{H-{1/2}}]
\]
and
\[
\vp^{(2)}_{u_h} = [ (s-u_{h})^{H-{1/2}}+u_{h}^{H-
{1/2}%
}] ^{2}.
\]
Moreover, the integral of $\prod_{h=2}^{k} \vp^{(2)}_{u_h}$ is easily
bounded: indeed, we have
\begin{eqnarray*}
&&\int_{u_{1}<u_2<\cdots<u_{k}<s}\prod_{h=2}^{k} \vp^{(2)}_{u_h}
\,du_{2}\cdots
du_{k}\\
&&\qquad\le\int_{u_{1}<u_2<\cdots<u_{k}<s}\prod_{h=2}^{k}[
(s-u_{h})^{H-{1%
}/{2}}+( u_{h}-u_{1}) ^{H-{1/2}}]
^{2}\,du_{2}\cdots
du_{k}\\
&&\qquad\le C \int_{ [u_1,s]^{k-1} }\prod_{h=2}^{k}[ (s-u_{h})^{
2H-1}+( u_{h}-u_{ 1}) ^{2H-1}] \,du_{2}\cdots
du_{k} \\
&&\qquad\le C(s-u_{1})^{2(k-1)H},
\end{eqnarray*}
with the convention $u_{k+1}=s$.
Therefore, plugging this inequality into (\ref{eq:bnd-J-0s-nu}) and
making the change of variables $s-u_1=v$ and $y=\frac v{t-s}$, we get
%
\begin{eqnarray*}
| J_{0s}(\nu)| &\leq& C\int_{0}^{s}[
(s-u_{1})^{H-{1/2}}-(t-u_{1})^{H-{1/2}}]
[ (s-u_{1})^{H-{1/2}}+u_{1}^{H-{1/2}}]\\
&&\hspace*{21.5pt}{}\times
(s-u_{1})^{2(k-1)H}\,du_{1} \\
&=&C\int_{0}^{s}[ v^{H-{1/2}}-(t-s+v)^{H-{1/2}}]
[ v^{H-{1/2}}+(s-v)^{H-{1/2}}]\\
&&\hspace*{21.5pt}{}\times v^{2(k-1)H}\,dv
\\
&=&C(t-s)^{2kH}\int_{0}^{s/(t-s)}[ y^{H-{1/2}}-(1+y)^{H-
{1/2%
}}]\\
&&\hspace*{91.7pt}{}\times
\biggl[ y^{H-{1/2}}+\biggl(\frac{s}{t-s}-y\biggr)^{H-{1/2}}\biggr]
y^{2(k-1)H}\,dy.
\end{eqnarray*}
We are now in a position to use Lemma \ref{lem2} with $A=s/(t-s)$, and
we obtain
%
%
\begin{equation} \label{eq:bnd-J-0s-nu-2}
| J_{0s}(\nu)| \le C(t-s)^{2kH},
\end{equation}
which implies that $J_{0s}(\nu)^2 \le C(t-s)^{2jH}$, owing to the
fact that $2k=j$.

\textit{Step} 2: \textit{suppose that $j_{k}=1$.} Then Proposition \ref
{prop:ito-strato} gives
\[
J_{0s}(\nu)=\int_{0<u_{1}<\cdots<u_{k}<s}\partial Z_{u_{1}}(1)\cdots
\partial Z_{u_{k-1}}(k-1)K(\tau_{j},u_{k})\,
\partial W_{u}(i_{j})
\]
and%
\begin{eqnarray*}
\be[J_{0s}(\nu)^{2}]&=&\int_{0}^{s}
\be( J_{0u}(\nu^{\prime})^{2})
K(\tau_{j},u)^{2}\,du \\
&\le& \int_{0}^{s}
\be( J_{0u}(\nu^{\prime})^{2})
\bigl( (s-u)^{2H-1} + u^{2H-1} \bigr)\,du,
\end{eqnarray*}
with $\nu^{\prime}=(j_{1},\ldots,j_{k-1})$. This relation allows us
to set an induction procedure,
as we shall see later.

\textit{Step} 3: \textit{suppose that $j_{k},j_{k-1},\ldots,j_{b+1}=2$ and
$j_b=1$, where $b\ge2$}.
We assume that $i_{m(h)} = i_{m(h)-1}$ for $h=b+1 ,\ldots, k$. Here
again, Proposition \ref{prop:ito-strato} implies
\begin{eqnarray*}
J_{0s}(\nu) &=&\int_{0<u_{1}<\cdots<u_{k}<s}\partial
Z_{u_{1}}(1)\cdots
\partial Z_{u_{b}}(b)\\
&&\hspace*{62.6pt}{}\times
\prod_{h=b+1}^{k}
K\bigl(\tau_{m(h)-1},u_{h}\bigr)K\bigl(\tau
_{m(h)},u_{h}\bigr) \,du_{1}\cdots du_{k},
\end{eqnarray*}
and Fubini's theorem yields
\[
J_{0s}(\nu) = \int_0^s J_{0u _b}(\nu')
K\bigl(\tau_{m(h)}, u_b\bigr) G(u_b) \,dW_{u_b}\bigl(i_{m(h)}\bigr),
\]
with $\nu'=(j_1, \ldots, j_{b-1})$, and where
\begin{eqnarray*}
G(u_b)&=&\int_{0<u_{b}< u_{b+1} <\cdots<u_{k}<s}
\prod_{h=b+1}^{k}
K\bigl(\tau_{m(h)-1},u_{h}\bigr)\\
&&\hspace*{114.7pt}{}\times K\bigl(\tau^{m(h)},u_{h}\bigr)\,du_{b+1}\cdots du_{k}.
\end{eqnarray*}
As for the previous bound (\ref{eq:bnd-J-0s-nu-2}) we obtain
\[
|G(u_b)| \le C (s-u_b)^{2(k-b)H}.
\]
Therefore
\begin{eqnarray*}
\be[J_{0s}(\nu)^2]&=& \int_0^s \be[J_{0u_b}(\nu')^2]
K\bigl(\tau_{m(h)}, u_b\bigr)^2 G(u_b)^2
\,du_b \\
&\le& C\int_0^s \be[J_{0u_b}(\nu'')^2]
[ (s-u_{b})^{2H-1%
}+u_{b}^{2H-1}] (s-u_b)^{4(k-b)H}
\,du_b .
\end{eqnarray*}
Notice that the above inequality includes the inequality obtained in
Step 2, which corresponds to the case $b=k$.

\textit{Step} 4: \textit{suppose that $j_{k},j_{k-1},\ldots,j_{b+1}=2$, $j_b=1$,
$j_{b-1}, j_{b-2}, \ldots, j_{c+1}=2$ and $j_c=1$, where $2\le c\le
b$. }
We assume also that $i_{m(h)} = i_{m(h)-1}$ for $h=c+1, \ldots, b-1,
b+1 ,\ldots, k$.
By the same arguments as in Step 2 we obtain
\begin{eqnarray*}
\be[J_{0s}(\nu)^2]
&\le& C\int_{0<u_c<u_b<s} \be[J_{0u_c}(\nu')^2]
[ (u_b-u_{c})^{2H-1%
}+u_{c}^{2H-1}]\\
&&\hspace*{58.3pt}{}\times (u_b-u_c)^{4(b-c)H}
[ (s-u_{b})^{2H-1%
}+u_{b}^{2H-1}]\\
&&\hspace*{58.3pt}{}\times (s-u_b)^{4(k-b)H}
\,du_c \,du_b ,
\end{eqnarray*}
with $\nu'=(j_1, \ldots, j_{c-1})$.
Replacing $u_{b}^{2H-1}$ by $(u_{b}-u_c)^{2H-1}$ and integrating with
respect to $u_b$ yields
\begin{eqnarray*}
\be[J_{0s}(\nu)^2]
&\le& C\int_0^s \be[J_{0u_c}(\nu')^2]
[ (s-u_{c})^{2H-1%
}+u_{c}^{2H-1}]\\
&&\hspace*{22.3pt}{}\times (s-u_c)^{4(k-c)H+2H} \,du_c .
\end{eqnarray*}

\textit{Step} 5: \textit{iteration scheme}. Iterating the argument in
Step 4, we
reduce the size of $\nu'$ until we obtain a multiindex of length $r$
such that $\nu'=(1,2,\ldots,2)$ or $\nu'=(2,2,\ldots,2)$, with
$j_{r+1}=1$, and we obtain an estimate of the form
%
%
\begin{eqnarray} \label{e2}
\be[J_{0s}(\nu)^2]
&\le& C\int_0^s \be[J_{0u }(\nu')^2]
[ (s-u )^{2H-1%
}+u ^{2H-1}]\nonumber\\[-8pt]\\[-8pt]
&&\hspace*{21.5pt}{}\times
(s-u )^{2H \sum_{ l=r+2}^k j_l } \,du.\nonumber
\end{eqnarray}
Suppose first that $\nu'=(1,2,\ldots,2)$. Then,
\begin{eqnarray*}
J_{0s}(\nu' ) &=& \int_{0<u_{1}<\cdots<u_{r}<u}
[K(t,u_1)-K(s,u_1)] \\
&&\hspace*{64pt}{}\times\prod_{h=2}^r K\bigl(\tau_{m(h)-1}, u_h\bigr)
K\bigl(\tau_{m(h)}, u_h\bigr) \,dW_{u_1}(i_1) \,du_2 \cdots du_r,
\end{eqnarray*}
and by Fubini's theorem
\[
J_{0s}(\nu' ) = \int_0^u
[K(t,u _1)-K(s,u_1 )] F(u_1) \,dW_{u_1}(i_1) ,
\]
where
\[
F(u_1)= \int_{u_{1}<u_2<\cdots<u_{r}<u}\prod_{h=2}^r K\bigl(\tau
_{m(h)-1}, u_h\bigr)
K\bigl(\tau_{m(h)}, u_h\bigr) \,du_2 \cdots du_r.
\]
As in the proof of (\ref{eq:bnd-J-0s-nu-2}) we get
\[
| F(u_1)| \le C(u-u_1) ^{2(r-1)H}.
\]
Therefore,
%
%
\begin{eqnarray} \label{e1a}
\be[J_{0s}(\nu' )^2]&\le& C \int_0^u
[(t-u_1)^{H-1/2} - (s-u_1)^{H-1/2}]^2\nonumber\\[-8pt]\\[-8pt]
&&\hspace*{22.4pt}{}\times (u-u_1) ^{ 4(r-1)H} \,du_1.\nonumber
\end{eqnarray}
Substituting (\ref{e1a}) into (\ref{e2}) yields, after integrating in
the variable $u$,
\[
\be[J_{0s}(\nu)^2]\le C\int_0^s [(t-u)^{H-1/2} -(s-u)
^{H-1/2}]^2 (s-u)^{2(j-1)H} \,du.
\]
Performing the changes of variables $v=s-u$ and $y=v/(t-s)$, we end up with
%
%
\begin{eqnarray}\label{eq:bnd-J-0s-nu-3}\qquad
\be[J_{0s}(\nu)^2]&\le& C(t-s)^{2jH}\int_0^{s/(t-s)}
[(1+t)^{H-1/2} -y ^{H-1/2}]^2 y^{2(j-1)H} \,dy \nonumber\\[-8pt]\\[-8pt]
&\le& C(t-s)^{2jH},\nonumber
\end{eqnarray}
where the last step is obtained thanks to a slight variation of Lemma
\ref{lem2}.

If $\nu'=(2,2,\ldots,2)$, then we proceed as in Step 1 and we obtain
%
%
\begin{eqnarray} \label{e3a}
| J_{0u}(\nu' ) | &\le& C \int_0^u [(s-u_1)^{H-1/2}
-(t-u_1) ^{H-1/2}]\nonumber\\
&&\hspace*{22.4pt}{}\times
[(u-u_1) ^{H-1/2} +u_1^{H-1/2} ]\\
&&\hspace*{22.4pt}{}\times (u-u_1)^{2(r-1)H}\,du_1.\nonumber
\end{eqnarray}
Substituting (\ref{e3a}) into (\ref{e2}), integrating first in the
variable $u$ and using the same arguments as in Step 1 we obtain also
the estimate
%
%
\begin{equation}\label{e5}
\be[J_{0s}(\nu)^2]\le C(t-s)^{2jH}.
\end{equation}

\textit{Step} 6: \textit{conclusion}. Our bounds (\ref
{eq:bnd-J-0s-nu-3}) and
(\ref{e5}) on $J_{0s}(\nu)$ yield the same kind of estimate for the
term $C_{st}^{j}$. Thus relation (\ref{eq:bnd-B-st-j}) gives
$B_{st}^{j}\lesssim(t-s)^{2nH}$. This estimate can now be plugged into
the definition (\ref{eq:def-Q1-Q2}) of $Q^2$, then in the definition of
$Q$, which leads to our claim (\ref{eq:def-Q}). The proof is now complete.
\end{pf}

\subsection{\texorpdfstring{Proof of Theorem \protect\ref{thm:main-thm}}{Proof of Theorem 1.1.}}

Before we prove our main theorem, we need a last elementary technical
ingredient, which relies on the notational
convention given at the beginning of the current section.
\begin{lemma}\label{lem:der-prod-K}
For $n\ge3$, $j=2,\ldots,n-1$ and $0\le s<t\le T$, set
\[
M_{st}^{n,j}=K_{s}^{\otimes(j-1)} \der K_{st} K_{t}^{\otimes(n-j)}.
\]
Recall that for an element $M\in\cac_2$, $\der M$ is defined by (\ref
{eq:simple_application}). Then
\begin{eqnarray*}
\der M_{sut}^{n,j}&=&
-\sum_{m=1}^{j-1} K_{s}^{\otimes(m-1)} \der K_{su}
K_{u}^{\otimes
(j-1-m)} \der K_{ut} K_{t}^{\otimes(n-j)} \\
&&{}+K_{s}^{\otimes(j-1)} \der K_{su} \sum_{m=1}^{n-j}
K_{u}^{\otimes
(m-1)} \der K_{ut} K_{t}^{\otimes(n-j-m)}.
\end{eqnarray*}
The relation still holds true for $j\in\{1,n\}$ and $n=2$, with the
convention $K^{\otimes0}=\mathbf{1}$ and $\der K^{\otimes0}=0$.
\end{lemma}
\begin{pf}
This proof is completely elementary, and included here for
the sake of completeness, since it uses heavily the notation of Section
\ref{sec:algebraic-vocabulary}.

First, if $a,b,c$ are 3 increments in $\cac_1$, and if we define $N\in
\cac_2$ by $N_{st}=a_s \der b_{st} c_{t}$, then a simple
application of
Definition (\ref{eq:simple_application}) gives
\[
\der N_{sut}= - \der a_{su} \der b_{ut} c_{t} + a_{s} \der
b_{su} \der c_{ut}.
\]
Our claim is thus proved by applying this relation to $a=K^{\otimes
(j-1)}$, $b=K$, $c=K^{\otimes(n-j)}$, and observing that $[\der
K^{\otimes l}]_{st}= \sum_{p=1}^{l} K_{s}^{\otimes(p-1)} \der
K_{st} K_{t}^{\otimes(l-p)}$.
\end{pf}
\begin{pf*}{Proof of Theorem \ref{thm:main-thm}}
The structure of the proof is the same as in the second order case of
Section \ref{sec:prop-second-intg}: we first reduce the algebraic
relations (\ref{eq:multiplicativity}) and (\ref{eq:geom-rough-path}) to
the case of some fixed $s,u,t$ by standard considerations. Then we
first focus on~(\ref{eq:multiplicativity}).

\textit{Step} 1: \textit{proof of the multiplicative property} (\ref
{eq:multiplicativity}). Fix
$(s,u,t)\in\cs_{3,T} $. Recall that $\hbb_{st}^{\bn,j}$ is defined by
(\ref{eq:def-hat-Bn-2-summarized}). Therefore, invoking Lemma
\ref{lem:der-prod-K}, $\der\hbb^{\bn,j}$ is given by
%
%
\begin{eqnarray}\label{eq:der-Bn-j}\quad
&&\der\hbb_{sut}^{\bn,j}(i_1,\ldots,i_n)\nonumber\\
&&\qquad=
(-1)^{j} \int_{A_{j}^{n}} \sum_{m=1}^{j-1} K_{s}^{\otimes(m-1)}
\der
K_{su} K_{u}^{\otimes(j-1-m)} \der K_{ut} K_{t}^{\otimes
(n-j)} \,dW \\
&&\qquad\quad{}+(-1)^{j-1} \int_{A_{j}^{n}} K_{s}^{\otimes(j-1)} \der K_{su} \sum
_{m=1}^{n-j} K_{u}^{\otimes(m-1)} \der K_{ut} K_{t}^{\otimes
(n-j-m)} \,dW.
\nonumber
\end{eqnarray}

On the other hand, set
$Z_{sut}=\sum_{n_{1}=1}^{n-1}\mathbf{B}_{su}^{\mathbf{n_{1}}}\mathbf
{B}_{ut}^{\mathbf{n-n_1}}$. One can easily check that
%
%
\begin{eqnarray}\label{eq:def-Z-sut}
Z_{sut}
&=&\sum_{n_{1}=1}^{n-1}\sum_{k=1}^{n_{1}}\sum_{h=1}^{n-n_{1}}\hat
{\mathbf{B%
}}_{su}^{\mathbf{n_{1}},k}\hat{\mathbf{B}}_{ut}^{\mathbf
{n-n_{1}},h}\nonumber\\
&=&\sum_{n_{1}=1}^{n-1}\sum_{k=1}^{n_{1}}%
\sum_{h=1}^{n-n_{1}}(-1)^{k+h}
\int_{A_{k,h}(n_1)} K_{s}^{\otimes(k-1)} \der K_{su}K_{u}^{\otimes(n_1-k+h-1)} \der K_{ut}\\
&&\hspace*{131.4pt}{}\times  K_{t}^{\otimes(n-n_1-h)}
\,dW, \nonumber
\end{eqnarray}
%
where $A_{k,h}(n_1)$ is the set defined by
\begin{eqnarray*}
A_{k,h}(n_1) &=& A_k^{n_1}\times A_h^{n-n_1}\\
&=&\{(u_{1},\ldots,u_{n});
u_{k}<u_{k+1}<\cdots<u_{n_{1}}, u_{k}<u_{k-1}<\cdots<u_{1}, \\
&&\hspace*{5.1pt}u_{n_1+h} <u_{n_1+h+1}<\cdots<u_{n}, u_{n_1+h}<u_{n_1+h-1}<\cdots
<u_{n_{1}+1}\}.
\end{eqnarray*}
We want to show that (\ref{eq:def-Z-sut}) and (\ref{eq:der-Bn-j}) coincide.

In order to follow the computations below, it might be useful to keep
in mind an illustration of the
coordinate ordering on a set of the form $A_{k,h}(m)$, for which an
example is provided at Figure \ref{Fig1} (note that the ordering
between $u_m$ and $u_{m+1}$ is not specified).

%
\begin{figure}[b]

\includegraphics{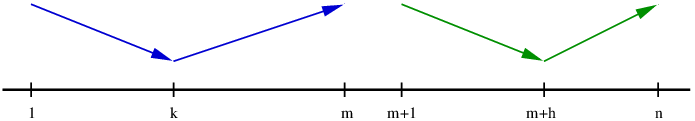}

\caption{Coordinates ordering on $A_{k,h}(m)$.}
\label{Fig1}
\end{figure}

Notice that on the set $A_{k,h}(n_1)\cap\{u_{k}<u_{n_1+h}\}$ the
minimum of the
coordinates is $u_{k}$, and on the set $A_{k,h}(n_1)\cap\{
u_{n_1+h}<u_{k}\}$ the
minimum is $u_{n_1+h}$. Define%
\[
A_{h,k}^{1}(n_1)=A_{k,h}(n_1)\cap\{u_{k}<u_{n_1+h}\}
\quad\mbox{and}\quad
A_{h,k}^{2}(n_1)=A_{k,h}(n_1)\cap\{u_{n_1+h}<u_{k}\}.
\]
Consider now the decomposition $Z=Z^{1}+Z^{2}$, where%
\begin{eqnarray*}
Z^{i}_{sut}&=&\sum_{n_{1}=1}^{n-1}\sum_{k=1}^{n_{1}}%
\sum_{h=1}^{n-n_{1}}(-1)^{k+h}\int_{A_{k,h}^{i} (n_1)}
K_{s}^{\otimes(k-1)} \der K_{su} K_{u}^{\otimes(n_1-k+h-1)}
\der K_{ut}\\
&&\hspace*{131.7pt}{}\times K_{t}^{\otimes(n-n_1-h)} \,dW.
\end{eqnarray*}
We fix $j$ and we try to compute the contribution of $Z^{i}_{sut}$ on
the set $%
A_{j}^{n}$ for $i=1,2$. This contribution will be the sum of the
integrals on
the set $A_{j}^{n}\cap A_{k,h}^{i}(n_1)$, for each $k=1,\ldots
,n_{1}$, $%
h=1,\ldots,n-n_{1}$ and for each $n_{1}=1,\ldots,n-1$.

Notice first\vspace*{2pt} that the intersection $A_{j}^{n}\cap A_{k,h}^{1}(n_1)$
is nonempty
only if $k=j$, \mbox{$h=1$} and $u_{n_1}<u_{n_1+1}$ which also implies $j\le
n_1$. Moreover, in this case we have $A_{j,1}^{1}(n_1) \cap\{ u_{n_1}
< u_{n_1+1} \}=A_{j}^{n}$. In this way we obtain that the
contribution of $Z^{1}_{sut}$ on $A_{j}^{n}$ is
%
%
\begin{eqnarray}\label{eq:contribution-Z1}\quad
&&(-1)^{j-1}\sum_{n_{1}=j}^{n-1}\int_{A_{j}^{n}}
K_{s}^{\otimes(j-1)} \der K_{su} K_{u}^{\otimes(n_1-j)} \der
K_{ut} K_t^{\otimes(n-n_1-1)} \,dW \nonumber\\[-8pt]\\[-8pt]
&&\qquad= (-1)^{j-1}\sum_{m=1}^{n-j}\int_{A_{j}^{n}}
K_{s}^{\otimes(j-1)} \der K_{su} K_{u}^{\otimes(m-1)} \der
K_{ut}
K_t^{\otimes(n-m-j)} \,dW,\nonumber
\end{eqnarray}
where we have used the simple change of variables $n_1-j=m-1$.
In the same manner, on the set $A_{j}^{n}\cap A_{k,h}^{2}(n_1)$ we
have $k=n_1$, $n_1+h=j$, which also implies $n_1\le j-1$. Therefore,
the contribution of $Z^{2}_{sut}$ on $A_{j}^{n}$ is%
%
%
\begin{equation}\label{eq:contribution-Z2}
(-1)^{j}\sum_{n_{1}=1}^{j-1}\int_{A_{j}^{n}}
K_{s}^{\otimes(n_1-1)} \der K_{su} K_{u}^{\otimes(j-1-n_1)}
\der K_{ut} K_{t}^{\otimes(n-j)} \,dW.
\end{equation}
One can now easily verify that the sum of (\ref{eq:contribution-Z1})
and (\ref{eq:contribution-Z2}) is equal to the term (\ref{eq:der-Bn-j}).

It remains to prove that the contribution of $Z_{sut}$ to the set
$(\bigcup_j A_{j}^{n})^c$ is zero. For this, observe that $(\bigcup_j
A_{j}^{n})^c$ can be\vspace*{1pt} split into slices $D_{k,p,h}$ of the following
form: for $1\le k\le p\le n-1$, we assume that $u_k < u_{k-1} < \cdots
< u_1$ and
$u_k <u_{k+1} <\cdots< u_{p}$ but $u_{p}>u_{p+1}$. Suppose also that
$1\le h\le n-p$ and that $u_{p+h}$ is the minimum of the coordinates
$u_{p+1} , \ldots, u_n$. Then, for $D_{k,p,h}$ to be a subset of
$
\bigcup_{n_1=1}^n \bigcup_{k,h} A_{k,h}(n_1),
$
we need the further condition $u_{p+h} < u_{p+h+1} <\cdots< u_n$ and
$u_{p+h} < u_{p+h-1} <\cdots< u_{p+1}$. With all these constraints in
mind, it is easily seen that $D_{k,p,h}$ corresponds to two possible
choices of set $A_{k,h}(n_1)$. Indeed, we have
\[
D_{k,p,h}=A_{k,h}(p)=A_{k,h+1}(p-1).
\]
Going back now to the expression (\ref{eq:def-Z-sut}) of $Z_{sut}$, it
is readily checked that the two contributions, respectively, on
$A_{k,h}(p)$ and $A_{k,h+1}(p-1)$, yield two terms with opposite sign,
which cancel out in the sum.

\textit{Step} 2: \textit{proof of the geometric property}
(\ref{eq:geom-rough-path}).
Fix $n,m$ such that $n+m\le\lfloor1/\ga\rfloor$ and let $(s,t)\in
\mathcal{S}_{2,T}$. Consider the product
\begin{eqnarray*}
&&\bb^{\bn}_{st}(i_1,\ldots,i_n) \bb^{\bmm}_{st}(j_1,\ldots,j_m
)\\
&&\qquad=\sum_{j=1}^n \sum_{h=1}^m (-1)^{j+h}
\biggl(\int_{A^n_j} K_s^{\otimes(j-1)} \delta K_{st} K_t^{\otimes(n-j)}
\,dW \biggr) \\
&&\hspace*{62.6pt}{}\times\biggl(\int_{A^m_h} K_s^{\otimes(h-1)} \delta K_{st}
K_t^{\otimes(m-h)} \,dW \biggr),
\end{eqnarray*}
where we have used notation (\ref{eq:def-hat-Bn-2-summarized}) and
where we recall that the sets $A^n_j$ and $A^m_h$ are defined by
\begin{eqnarray*}
A^n_j&=&\{u\in[0,t]^n\dvtx u_j<u_{j-1}<\cdots<u_1, u_j< u_{j+1}<\cdots
<u_n\},
\\
A^m_h&=&\{v\in[0,t]^m\dvtx v_h<v_{h-1}<\cdots<v_1, v_h< v_{h+1}<\cdots
<v_m\}.
\end{eqnarray*}
The product of the two Stratonovich integrals can be expressed as a
Stratono\-vich integral on the region
$A^n_j \times A^m_h$ with respect to the differential
\[
dW_{u_1}(i_1) \cdots dW_{u_n}(i_n)\,dW_{v_1}(j_1) \cdots dW_{v_m}(j_m).
\]
We will make use of the notation $z=(u,v)$, where $z_\alpha=u_\alpha$,
for $\alpha=1,\ldots, n$ and $z_\alpha= v_{\alpha-n}$ for $\alpha=n+1,
\ldots, n+m$.
As in Step 1, the region $A^n_j \times A^m_h$ can be first decomposed
into the union of the disjoint regions $D_{j,h}$ and $E_{j,h}$,
corresponding, respectively, to the additional constraints $\{u_j<v_h\}$
and $\{u_j>v_h\}$ (notice that this decomposition is valid
up to the set $\{u_j=v_h\}$, whose contribution to the stochastic
integral is null).

Consider first the case $\{u_j<v_h\}$. On $D_{j,h}$ the minimum of all
the coordinates $z_\alpha$ is $z_j$.
Then $D_{j,h}$ can be
further decomposed into the disjoint union of the sets
\begin{eqnarray*}
D_{j,h,1}^\pi& =& \{z\in[0,t]^{n+m}\dvtx z_j<z_{\alpha_{j+h-2}} <\cdots<
z_{\alpha_{1}},\\
&&\hspace*{50.6pt}
z_j<z_{\beta_1} <\cdots< z_{\beta_{n-j+1+m-h}}\} \\
&&{} \cap\{z_{n+h} <z_{n+h-1}\},
\end{eqnarray*}
where
\[
\pi( 1,\ldots, n+m)= (\alpha_1, \ldots, \alpha_{j+h-2}, j, \beta_1,
\ldots
, \beta_{n-j+1+m-h})
\]
runs over all permutations of the coordinates $1,\ldots, n+m$ such that
$\pi(j+h-1)=j$ and:

\mbox{}\hphantom{i}(i)
$\alpha_1, \ldots, \alpha_{j+h-2}$ is a permutation of the coordinates
$1,\ldots,j-1$ and $n+1, \ldots, n+h-1$ that preserves the orderings of
the indices $ 1,\ldots, j-1 $ and $ n+1, \ldots, n+h-1 $.

(ii)
$\beta_1, \ldots, \beta_{n-j+1+m-h}$ is a permutation of the
coordinates $j+1, \ldots,n$ and $n+h, \ldots, n+m$ that preserves the
orderings of the indices $ j+1,\ldots, n $ and $ n+h, \ldots, n+m $.

Notice that $\alpha$ is the inverse of a shuffle since it splits an
ordered list into two ordered sublists. The same remark applies to
$\beta$.

Moreover, $D_{j,h}$ can be also
be decomposed into the disjoint union of the sets
%
\begin{eqnarray*}
D_{j,h,2}^{\tpi} &=& \{z\in[0,t]^{n+m}\dvtx z_j<z_{\alpha_{j+h-1}}
<\cdots
< z_{\alpha_1},\\
&&\hspace*{58.8pt}
z_j<z_{\beta_1} <\cdots< z_{\beta_{n-j+m-h}}\} \\
&&{} \cap\{z_{n+h} <z_{n+h+1}\},
\end{eqnarray*}
where
\[
\tpi(1,\ldots, n+m)= (\alpha_1, \ldots, \alpha_{j+h-1}, j, \beta_1,
\ldots
, \beta_{n-j +m-h})
\]
runs over all permutations of the coordinates $1,\ldots, n+m$ such that
$\tpi(j+h)=j$ and:

\mbox{}\hphantom{i}(i)
$\alpha_1, \ldots, \alpha_{j+h-1}$ is a permutation of the coordinates
$1,\ldots,j-1$ and $n+1, \ldots, n+h$ that
preserves the orderings of the indices $ 1,\ldots, j-1 $ and $ n+1,
\ldots
, n+h $.

(ii)
$\beta_1, \ldots, \beta_{n-j +m-h}$ is a permutation of the coordinates
$j+1, \ldots,n$ and $n+h+1, \ldots, n+m$ that preserves the orderings of
the indices $ j+11,\ldots, n $ and $ n+h+1, \ldots, n+m $.

Then, on the set $D_{j,h}$ we write
\begin{eqnarray*}
&& K_s^{\otimes(j-1)} \delta K_{st} K_t^{\otimes(n-j)}K_s^{\otimes
(h-1)} \delta K_{st} K_t^{\otimes(m-h)} \\
&&\qquad =K_s^{\otimes(j-1)} \delta K_{st} K_t^{\otimes(n-j)}K_s^{\otimes
(h-1)} K_t^{\otimes(m-h+1)} \\
&&\qquad\quad{}-K_s^{\otimes(j-1)} \delta K_{st}
K_t^{\otimes(n-j)}K_s^{\otimes h} K_t^{\otimes(m-h)},
\end{eqnarray*}
and the integral
\[
I_{j,h}:= \int_{D_{j,h}} K_s^{\otimes(j-1)} \delta K_{st}
K_t^{\otimes
(n-j)}K_s^{\otimes(h-1)} \delta K_{st} K_t^{\otimes(m-h)} \,dW
\]
can be expressed as the sum $I_{j,h}=I_{j,h}^{+}+I_{j,h}^{-}$, with
\begin{eqnarray*}
I_{j,h}^{+}&=&
\sum_{\pi} \int_{D^\pi_{j,h,1}} (-1)^{j+h-2} \\
&&\hspace*{41.1pt}{}\times\prod_{l=1}^{j+h-2}
K(s,z_{\alpha_l}) \delta K_{st}(z_j) \\
&&\hspace*{41.1pt}{}\times
\prod_{l= 1}^{n-j+1+m-h} K(t,z_{\beta_l}) \,dW_{z_1}(i_1) \cdots
dW_{z_{n+m}}(i_{n+m})
\end{eqnarray*}
and
\begin{eqnarray*}
I_{j,h}^{-}&=&
\sum_{\tpi} \int_{D^{\tpi}_{j,h,2}} (-1)^{j+h-1} \\
&&\hspace*{41.1pt}{}\times\prod_{l=1}^{j+h-1}
K(s,z_{\alpha_l})
\delta K_{st}(z_j) \\
&&\hspace*{41.1pt}{}\times\prod_{l=1}^{n-j+m-h} K(t,z_{\beta_l}) \,dW_{z_1}(i_1) \cdots
dW_{z_{n+m}}(i_{n+m}).
\end{eqnarray*}

Let us handle first the term $I_{j,h}^{+}$: consider the permutation
$\sigma=\pi^{-1}$ of $1, \ldots, n+m$ which maps $\alpha_1, \ldots,
\alpha
_{j+h-2}$ into $1,\ldots, j+h-2$ and $\beta_1, \ldots,\break\beta
_{n-j+1+m-h}$ into
$j+h, \ldots, n+m$, with the additional condition $\sigma(j)=j+h-1$.
If we make this permutation in the coordinates of $I_{j,h}^{+}$ we obtain
\begin{eqnarray*}
I_{j,h}^{+}&=&
\int_{A_{j+h-1}^{n+m}\cap\{z_\nu<z_\eta\}} (-1)^{j+h-2} \\
&&\hspace*{65.1pt}{}\times\prod
_{l=1}^{j+h-2} K(s, z_l) \delta K_{st}(z_{j+h-1} ) \\
&&\hspace*{65.1pt}{}\times\prod_{l=j+h}^{n+m} K(t, z_l)
\,dW_{z_1} (k_1) \cdots dW_{z_{n+m}} (k_{n+m}),
\end{eqnarray*}
where $k_1, \ldots, k_{n+m}$ is a permutation of the indexes $i_1,
\ldots
, i_{n }, j_1, \ldots, j_m$ defined by
$k_\ell=i_{\sigma(\ell)}$ if $1\le\sigma(\ell)\le n$ and $k_\ell
=j_{\sigma(\ell)}$ if $n+1\le\sigma(\ell)\le n+m$, and where $\nu
,\eta
$ are defined by
\begin{eqnarray*}
\nu&=&\min\bigl\{ i\ge j+h\dvtx k_i \in\{j_1,\ldots, j_m\}\bigr\},\\
\eta&=&\max\bigl\{ i\le j+h-2\dvtx k_i \in\{j_1,\ldots, j_m\}\bigr\}.
\end{eqnarray*}

In the same way, we can consider a permutation $\sigma=\tilde{\pi}
^{-1}$ in the coordinates $z_i$ which maps $\alpha_1, \ldots, \alpha
_{j+h-1}$ into $1,\ldots, j+h-1$ and $\beta_1, \ldots,\beta
_{n-j+m-h}$ into
$j+h+1, \ldots, n+m$, and $\sigma(j)=j+h$. If we make this permutation
in the coordinates of $I_{j,h}^{-}$ we obtain
\begin{eqnarray*}
I_{j,h}^{-}&=&
\int_{A_{j+h}^{n+m}\cap\{z_\nu> z_\eta\}} (-1)^{j+h-1}\\
&&\hspace*{57.8pt}{}\times \prod
_{l=1}^{j+h-1} K(s, z_l) \delta K_{st}(z_{j+h-1} ) \\
&&\hspace*{57.8pt}{}\times \prod_{l=j+h+1}^{n+m} K(t, z_l)
\,dW_{z_1} (k_1) \cdots dW_{z_{n+m}} (k_{n+m}),
\end{eqnarray*}
where again $k_1, \ldots, k_{n+m}$ is a permutation of the indexes $i_1,
\ldots, i_{n }, j_1, \ldots, j_m$ defined by
$k_\ell=i_{\sigma(\ell)}$ if $1\le\sigma(\ell)\le n$ and $k_\ell
=j_{\sigma(\ell)}$ if $n+1\le\sigma(\ell)\le n+m$, and where $\nu
,\eta
$ are now defined by
\begin{eqnarray*}
\nu&=&\min\bigl\{ i\ge j+h+1\dvtx k_i \in\{j_1,\ldots, j_m\}\bigr\},\\
\eta&=&\max\bigl\{ i\le j+h-1\dvtx k_i \in\{j_1,\ldots, j_m\}\bigr\}.
\end{eqnarray*}

When we sum these integrals over all permutations $\sigma$ of the
above type, that is $\sigma=\pi^{-1}$ or $\sigma=\tilde{\pi} ^{-1}$,
and also over
$j$ and $h$, we obtain $\sum_{\bar k\in \mathrm{Sh}(\bar\imath
,\bar\jmath)} \bb^{\bn+\bmm,1}_{st}\break (k_1,\ldots,$ $k_{n+m})$,
where
\begin{eqnarray*}
&&\bb^{\bn+\bmm,1}_{st}(k_1,\ldots,k_{n+m})\\
&&\qquad= \sum_{p=1, k_p\in\{
i_1,\ldots, i_n\}}^{n+m}\int_{A_{p}^{n+m} } (-1)^{p-1} \\
&&\qquad\quad\hspace*{89.5pt}{}\times \prod_{l=1}^{p-1}
K(s, z_l) \delta K_{st}(z_{p} ) \\
&&\qquad\quad\hspace*{89.5pt}{} \times\prod_{l=p+1}^{n+m} K(t, z_l)
\,dW_{z_1} (k_1) \cdots dW_{z_{n+m}} (k_{n+m}).
\end{eqnarray*}

In a similar manner we could show that the sum of the integrals over
$E_{j,h}$ give rise to
$\sum_{\bar k\in \mathrm{Sh}(\bar\imath,\bar\jmath)} \bb
^{\bn+\bmm
,2}_{st}(k_1,\ldots,k_{n+m})$, for $h=1,\ldots, m$,
where
\begin{eqnarray*}
&&\bb^{\bn+\bmm,2}_{st}(k_1,\ldots,k_{n+m})\\
&&\qquad= \sum_{p=1, k_p\in\{
j_1,\ldots, j_m\}}^{n+m}\int_{A_{p}^{n+m} } (-1)^{p-1} \\
&&\qquad\quad\hspace*{93.1pt}{}\times\prod_{l=1}^{p-1}
K(s, z_l) \delta K_{st}(z_{p} ) \\
&&\qquad\quad\hspace*{93.1pt}{}\times\prod_{l=p+1}^{n+m} K(t, z_l)
\,dW_{z_1} (k_1) \cdots dW_{z_{n+m}} (k_{n+m}).
\end{eqnarray*}
Taking into account the two contributions $\bb^{\bn+\bmm,1}_{st}$ and
$\bb^{\bn+\bmm,2}_{st}$, the proof of the geometric property is now
easily finished.

\textit{Step} 3: \textit{proof of the regularity property}.
As in Proposition \ref{prop:regularity-B2}, the fact that $\bb^{\bn}$
belongs to $\cac_2^{n\ga}$ for any $\ga<H$ is an easy consequence of
the moment estimate of Proposition \ref{prop:moments-B-n}, plus a
simple induction procedure.

Indeed, assume that $\bb^{\bk}\in\cac_2^{k\ga}( (\mathbb
{R}^{d})^{\otimes k})$ for any $k\le n-1$. Then Lemma \ref{GRR-2} gives
here that $\cn[\bb^\bn;\cac_2^{n\ga}( \mathbb{R}^{d^2})]\lesssim
A+D$, with
\[
A= \biggl(\int_{\cs_{2,T}} \frac{| \bb^{\bn}_{uv}
|^{2p}}{|u-v|^{2n \ga
p+4}} \,du \,dv \biggr)^{ 1/({2p})}
\quad\mbox{and}\quad
D= \cn[ \delta\bb^{\bn}; \cac_3^{n\ga}( \mathbb{R}^{d^2})
].
\]
Furthermore, since we have seen that $\bb^\bn$ satisfies the
multiplicative property (\ref{eq:multiplicativity}), then $D$ is easily
shown to be almost surely finite thanks to our induction hypothesis.
Finally, the quantity $\be[A]$ can be bounded along the same lines as
in Proposition \ref{prop:regularity-B2}, except that Proposition \ref
{prop:moments-B-n} is used instead of Proposition \ref
{prop:bnd-second-moment-B2-st}.
\end{pf*}

\section{Relationship with other iterated integrals}
\label{sec:relation-other}

This section is devoted to a comparison of the rough path above fBm we
have just constructed with other existing iterated integrals. We first
treat the case of canonical (or pathwise) integrals defined in \cite
{CQ,FV-bk}, focusing on the double iterated integral case. Then we
shall try to replace our construction into the general context of
Fourier normal ordering as introduced in \cite{Un09b}.

\subsection{Comparison with the canonical double iterated integral}

Consider $1/4<H<1$. We wish to compare $\bb^{\mathbf{2}}$ defined by (\ref
{eq:def-B2}) with the increment $\bb^{\mathbf{2},\mathrm{p}}$, where
%
%
\begin{equation}\label{eq:def-levy-pathwise}
\bb^{\mathbf{2},\mathrm{p}}_{st}:=\int_{s<u_1<u_2<t} dB_{u_1}(i_1) \,dB_{u_2}(i_2)
\end{equation}
is interpreted in the following way:

\mbox{}\hphantom{ii}(i) If $1/2<H<1$, $\bb^{\mathbf{2},\mathrm{p}}_{st}$ is defined in the
Young sense (or equivalently in the Stratonovich sense of Malliavin
calculus--see \cite{Nu06}).

\mbox{}\hphantom{i}(ii) If $H=1/2$, $\bb^{\mathbf{2},\mathrm{p}}_{st}$ corresponds to a
Stratonovich integral with respect to Brownian motion.

(iii) When $1/4<H<1/2$, $\bb^{\mathbf{2},\mathrm{p}}_{st}$ is defined by
a limiting procedure in \cite{CQ,FV-bk}, but is also shown in \cite{CQ}
to correspond to a Stratonovich integral in the Malliavin calculus
sense.

In all those cases, $\bb^{\mathbf{2},\mathrm{p}}$ can thus be defined thanks to
Malliavin calculus tools, and is also thought of as the canonical
double iterated integral for $B$. We shall keep this definition in mind
in the sequel, and refer to \cite{Nu06} for further definitions of
Malliavin calculus. Notice that ``p'' in in our notation $\bb^{\mathbf{2}
,\mathrm{p}}$ stands for pathwise.

Our comparison result for double iterated integrals can be read as follows:
\begin{proposition}\label{prop:comp-double}
Consider a $d$-dimensional fBm $B$ with Hurst index $1/4<H<1$. Let $\bb
^{\mathbf{2}}$ be the increment defined by (\ref{eq:def-B2}), and $\bb^{\mathbf{2}
,\mathrm{p}}$ defined by (\ref{eq:def-levy-pathwise}). For $0<b<a<t$,
set $\psi
_t(a,b)=\int_{a}^{t} K(v,a) \partial_{v}K(v,b) \,dv$. Then for $H\in
(1/4,1)\setminus\{1/2\}$, we have $\bb^{\mathbf{2}}-\bb^{\mathbf{2},\mathrm{p}}=\der f$,
where $f\dvtx\R_+\to\R^{d^2}$ is the process defined by
%
%
\begin{eqnarray}\label{eq:def-f-t}
&&f_t(i_1,i_2)\nonumber\\
&&\qquad=
\int_{0<u_1<u_2<t} \psi_{t}(u_2,u_1) \,dW_{u_1}(i_1) \,dW_{u_2}(i_2)\\
&&\qquad\quad{}-\int_{0<u_2<u_1<t} \psi_{t}(u_1,u_2) \,dW_{u_1}(i_1)
\,dW_{u_2}(i_2). \nonumber
\end{eqnarray}
In particular, $f(i_1,i_2)\equiv0$ if $i_1=i_2$. For $H=1/2$, one gets
the relation $\bb^{\mathbf{2}}-\bb^{\mathbf{2},\mathrm{p}}=0$.
\end{proposition}
\begin{remark}
Consider the antisymmetric parts $\bb^{\mathbf{2},a}$ and $\bb^{\mathbf{2},\mathrm{p},a}$
of $\bb^{\mathbf{2}}$ and $\bb^{\mathbf{2},\mathrm{p}}$, respectively, considered as
matrix-valued increments. These objects are usually referred to as L\'
{e}vy areas of $B$. Then it is readily checked that $\bb^{\mathbf{2},a}-\bb
^{\mathbf{2}
,\mathrm{p},a}=\der f$ as well.
\end{remark}
\begin{pf*}{Proof of Proposition \ref{prop:comp-double}}
It is easily shown, thanks to Proposition \ref{prop:algebraic-B2}, that
$\der\bb^{\mathbf{2}}=\der\bb^{\mathbf{2},\mathrm{p}}$. We thus know that $\bb^{\mathbf{2}
}-\bb
^{\mathbf{2},\mathrm{p}}=\der f$ for a certain function $f\in\cac_{1}$.
Furthermore, a possible choice for $f$ (unique up to constants) is simply
\[
f_t=\bb^{\mathbf{2}}_{0t}-\bb^{\mathbf{2},\mathrm{p}}_{0t}.
\]
We shall try to simplify the latter expression, and distinguish 3 cases:

\textit{Case} 1: $H=1/2$. In this situation the computations differ
slightly from the case $1/4<H<1/2$, since in $K_t(u)=\mathbf{1}_{[0,t]}(u)$
instead of the expression given by (\ref{eq:def-K}). However, the
relation $\bb^{\mathbf{2}}-\bb^{\mathbf{2},\mathrm{p}}=0$ is easily verified directly.

\textit{Case} 2: $1/2<H<1$. We treat this situation first, since it is
technically simpler than the rougher case $H<1/2$. The kernel $K$ is
given here by \cite{Nu06}, equation (5.8), instead of (\ref{eq:def-K}),
but still satisfies a relation of the form (\ref{eq:bnd-kernel}), which
allows to translate many of the bounds in Section \ref
{sec:intg-order-2}. In particular, both increments $\bb^{\mathbf{2},\mathrm{p}}$
and $\bb^{\mathbf{2}}$ are well defined. However, when $H>1/2$ we cannot assume
$\der f:=\bb^{\mathbf{2}}-\bb^{\mathbf{2},\mathrm{p}}$ lies in~$\cac_{1}^{2\ga}$, since
Lemma \ref{lem:bnd-K-2} cannot be applied anymore (additionally, $f\in
\cac_{1}^{2\ga}$ would mean $f\equiv\mathrm{Constant}$). We shall thus
only work with $f\in\cac_{1}^{\ga}$.

In order to find an amenable expression for $f$, decompose again $\bb
^{\mathbf{2}}_{0t}$ into $\hbb^{\mathbf{2},1}_{0t}+\hbb^{\mathbf{2},2}_{0t}$. Thanks to the
fact that $K(0,\cdot)\equiv0$, it is then easily seen from equation
(\ref{eq:def-B2-2}) that $\hbb^{\mathbf{2},2}_{0t}=0$. Thus, reading (\ref
{eq:def-B2-2}) in our particular situation yields
%
%
\begin{equation}\label{eq:bb2-0t}
\bb^{\mathbf{2}}_{0t}(i_1,i_2)=\int_{u_1<u_2} K_t(u_1) K_t(u_2)
\,dW_{u_1}(i_1) \,dW_{u_2}(i_2).
\end{equation}

For $H>1/2$, a suitable expression for $\bb^{\mathbf{2},\mathrm{p}}_{0t}$,
obtained by means of a Fubini-type arguments, is
\begin{eqnarray*}
\bb^{\mathbf{2},\mathrm{p}}_{0t}(i_1,i_2)&=&
\int_{0}^{t} \biggl(\int_{u_2}^{t} \partial_{v} K_{v}(u_2) B_v(i_1) \,dv
\biggr)\,dW_{u_2}(i_2) \\
&=&
\int_{[0,t]^2} \biggl(\int_{u_1\vee u_2}^{t} \partial_{v} K_{v}(u_2)
K_{v}(u_1) \,dv \biggr)
\,dW_{u_1}(i_1) \,dW_{u_2}(i_2)\\
:\!&=&J_{st}^{1}+J_{st}^{2},
\end{eqnarray*}
where
\begin{eqnarray*}
J_{st}^{1}&=&
\int_{0<u_1<u_2<t} \biggl(\int_{u_2}^{t} \partial_{v} K_{v}(u_2)
K_{v}(u_1) \,dv \biggr)
\,dW_{u_1}(i_1) \,dW_{u_2}(i_2), \\
J_{st}^{2}&=&
\int_{0<u_2<u_1<t} \biggl(\int_{u_1}^{t} \partial_{v} K_{v}(u_2)
K_{v}(u_1) \,dv \biggr)
\,dW_{u_1}(i_1) \,dW_{u_2}(i_2).
\end{eqnarray*}
Owing to a simple integration by parts argument, we have
\[
\int_{u_2}^{t} \partial_{v} K_{v}(u_2) K_{v}(u_1) \,dv
= K_t(u_1) K_t(u_2) - \int_{u_2}^{t} K_{v}(u_2) \partial
_{v}K_{v}(u_1) \,dv,
\]
and hence
\begin{eqnarray*}
J_{st}^{1}&=&
\int_{0<u_1<u_2<t}
\biggl[K(t,u_1) K(t,u_2)\\
&&\hspace*{53.5pt}{} - \int_{u_2}^{t} K_{v}(u_2) \partial
_{v}K_{v}(u_1) \,dv \biggr]
\,dW_{u_1}(i_1) \,dW_{u_2}(i_2)\\
&=&\bb^{\mathbf{2}}_{0t}(i_1,i_2)
-\int_{0<u_1<u_2<t}
\biggl(\int_{u_2}^{t} K_{v}(u_2) \partial_{v}K_{v}(u_1) \,dv \biggr)
\,dW_{u_1}(i_1) \,dW_{u_2}(i_2).
\end{eqnarray*}
Gathering all the expressions we have obtained so far and recalling our
notation $\psi_t(a,b)=\int_{a}^{t} K(v,a)$ $\partial_{v}K(v,b) \,dv$
for $0<b<a<t$, the proof of (\ref{eq:def-f-t}) is now readily completed.

\textit{Case} 3: $1/4<H<1/2$. Many of the computations of Case 2 can be
reproduced here, and we will just outline the main differences.

Since $H<1/2$, Lemma \ref{lem:bnd-K-2} and the results in \cite
{CQ,FV-bk} assert that $f$ is an element of $\cac_1^{2\ga}$ in the
current situation. Moreover, (\ref{eq:bb2-0t}) is still valid for
$H<1/2$, so that we only have to find an alternative expression for
$\bb
^{\mathbf{2},\mathrm{p}}_{0t}$.

Thanks to expression (5.29) in \cite{Nu06}, one can write
\[
\bb^{\mathbf{2},\mathrm{p}}_{0t}(i_1,i_2)=
\int_{0}^{t} [K_t^*B(i_1) ]_{u_2} \,dW_{u_2}(i_2):=L_{st}^{1}+L_{st}^{2},
\]
where
\begin{eqnarray*}
L_{st}^{1}&=&
\int_{0}^{t} \biggl(\int_{u_2}^{t} \partial_{v}K_v(u_2) \der
B_{u_2v}(i_1) \,dv\biggr)\,dW_{u_2}(i_2), \\
L_{st}^{2}&=&
\int_{0}^{t} K_t(u_2) B_{u_2}(i_1) \,dW_{u_2}(i_2).
\end{eqnarray*}
Then the same kind of arguments as for Case 2 (Fubini-type relations
and integration by parts for $K$) yield
$L_{st}^{1}=L_{st}^{11}+L_{st}^{12}$, with
\begin{eqnarray*}
L_{st}^{11}&=&
\int_{0<u_1<u_2<t} [K_t(u_2) \der K_{tv}(u_1)-\psi_t(u_2,u_1) ]
\,dW_{u_1}(i_1)
\,dW_{u_2}(i_2), \\
L_{st}^{12}&=&
\int_{0<u_2<u_1<t} \psi_t(u_1,u_2) \,dW_{u_1}(i_1) \,dW_{u_2}(i_2).
\end{eqnarray*}
It is also easily checked that
\[
L_{st}^{2}= \int_{0<u_1<u_2<t} K_t(u_2) K_{u_2}(u_1)
\,dW_{u_1}(i_1) \,dW_{u_2}(i_2).
\]
Recalling then $\bb^{\mathbf{2},\mathrm
{p}}_{0t}(i_1,i_2)=L_{st}^{11}+L_{st}^{12}+L_{st}^{2}$ we end up, after
some elementary algebraic manipulations, with expression (\ref{eq:def-f-t}).
\end{pf*}

\subsection{Comparison with the construction by Fourier normal ordering}

It is impossible to reproduce here the elegant formalism on which \cite
{Un09b} is based. We will thus just content ourselves with giving some
hints on the possibility to link our construction with the general
Fourier normal ordering program described in the latter reference.

One of the starting points in \cite{Un09b} is that any iterated
integral with respect to a function $X$ can be encoded by a tree whose
vertices are decorated by $\{1,\ldots,d\}$ if $X$ is $\R^d$-valued. A
Hopf algebra structure is usually added to this set of trees after the
pioneering work of Connes and Kreimer \cite{CK}, the resulting
structure being denoted by $\bh$.

In case of a smooth function $X$, consider $\bx^{\bn}(i_1,\ldots,i_n)$
defined by (\ref{eq:def-Xn-intro}) in the Riemann sense. Let also $\si
\in\Sigma_n$ be a permutation of $\{1,\ldots,n\}$. When one wishes to
express $\bx^{\bn}(i_{\si(1)},\ldots,i_{\si(n)})$ in terms of integrals
involving the indices $i_1,\ldots,i_n$ in this exact order, one is
naturally led to use operations on trees and forests, encoded in the
Hopf algebra structure alluded to above. After a huge amount of
formalization explained in \cite{Un09b}, this allows us to write, for
$0\le s<t\le T$,
%
%
\begin{equation}\label{eq:dcp-character}
\bx^{\bn}_{st}(i_1,\ldots,i_n)=
[(\chi_X^s\circ S) * \chi_X^t ](\T_n),
\end{equation}
where $\T_n$ designates the trunk tree of order $n$ decorated by
$i_1,\ldots,i_n$, $\chi_X^s$ is a character defined on $\bh$, $S$
stands for the antipode operation characteristic of Hopf algebras and
$*$ is a certain convolution product defined on $\bh$. Notice that the
equivalent of decomposition (\ref{eq:dcp-character}) in \cite{Un09b}
involves some so-called \textit{skeleton integrals}, which refer to
Fourier transform techniques. Our character $\chi_X^s$ is defined in
direct coordinates, in concordance with the Volterra-type
representation we have chosen.

Still in case of a smooth function $X$, a further analysis of the terms
$\chi_X^s$ allows the decomposition (valid for a multiindex
$(j_1,\ldots
,j_n)$ assimilated with its associated trunk tree)
%
%
\begin{equation}\label{eq:dcp-forest}
\chi_X^s(j_1,\ldots,j_n)=\sum_{\si\in\Sigma_n} I_s(\T^{\si}),
\end{equation}
where $\T^{\si}$ is a forest called permutation graph (see \cite{Un09b},
Lemma 1.5). This kind of decomposition is the one which has to be
generalized to nonsmooth situations. In our context, $I_s(\T^{\si})$
is obviously a Wiener multiple integral weighted by the kernel~$K$,
whose generic form is given by
\[
I_s(\T^{\si})=\int_{u_{\si(1)}<\cdots<u_{\si(n)}} \prod_{j=1}^{n}
K_{a_j}(u_j) \,dW_{u_j}(i_j),
\]
where each $a_j=s$ or $t$ according to the permutation graph under
consideration.

The algorithm set up in \cite{Un09b} in order to cope with nonsmooth
situations basically replaces the integrals $I_s(\T^{\si})$ for any
$\T
^{\si}$ having more than two vertices by something smoothed in Fourier
coordinates. Our approach is simpler (and rougher), in the sense that
we replace all those integrals by 0. We are thus just left with the
permutation graph $\T^{\si_0}$ corresponding to $\si_0\dvtx(1,\ldots
,n)\mapsto(n,\ldots,1)$, which is the only one containing trees
reduced to a root (see \cite{FU} for further explanations). It can then
be shown that, reading \cite{Un09b}, Lemma 3.6, in this context leads to
our definition (\ref{eq:def-Bn-2}) of the multiple iterated integral
with respect to $B$. In a sense, our construction is thus included in
the broader context of \cite{Un09b}. Nevertheless, let us insist on the
fact that we provide a simple and direct alternative approach to the problem.


%
\printaddresses

\end{document}